\documentclass[11pt]{article}
\usepackage[margin=1.5in]{geometry}
\usepackage{amsmath,amssymb}    
\usepackage{url}                
\usepackage{graphicx}           
\usepackage{amsthm}
\numberwithin{equation}{section}
\newtheorem{theorem}{Theorem}
\newtheorem{lemma}{Lemma}
\newtheorem{proposition}{Proposition}

\newtheorem{definition}{Definition}

\newtheorem{conjecture}[equation]{Conjecture}

\newcommand{\RR}{\mathbb{R}}

\newcommand{\cS}{\mathcal{S}}

\newcommand{\cH}{\mathcal{H}}

\newcommand{\one}{\mathbf{1}}

\newcommand{\inprod}[1]{\langle #1 \rangle}

\DeclareMathOperator*{\E}{\mathbb{E}}
\let\Pr\relax
\DeclareMathOperator*{\Pr}{\mathbb{P}}
\DeclareMathOperator*{\argmax}{\mathrm{argmax}}
\DeclareMathOperator*{\argmin}{\mathrm{argmin}}

\newcommand{\HSBM}{\mathsf{HSBM}}

\newcommand\diag{\mathrm{diag}}

\usepackage{mathabx}
\makeatletter
\newcommand\incircbin
{%
  \mathpalette\@incircbin
}
\newcommand\@incircbin[2]
{%
  \mathbin%
  {%
    \ooalign{\hidewidth$#1#2$\hidewidth\crcr$#1\ovoid$}%
  }%
}
\newcommand{\oeq}{\incircbin{=}}
\makeatother

\begin{document}

\title{Stochastic Block Model for Hypergraphs: Statistical limits and a semidefinite programming approach}

\author{
Chiheon Kim\thanks{MIT, Department of Mathematics, 77 Massachusetts Ave., Cambridge, MA 02139. Partially supported by ONR Grant N00014-17-1-2177.}\\
\texttt{chiheonk@math.mit.edu}
\and
Afonso S. Bandeira\thanks{Department of Mathematics, Courant Institute of Mathematical Sciences and Center for Data Science, New York University, New York, NY 10012. Afonso S. Bandeira was partially supported by NSF grants DMS-1317308, DMS-1712730, and DMS-1719545. Part of this work was done while Afonso S. Bandeira was with the Mathematics Department at MIT.}\\
\texttt{bandeira@cims.nyu.edu}
\and
Michel X. Goemans\thanks{MIT, Department of Mathematics, Room 2-474, 77 Massachusetts Ave., Cambridge, MA 02139. Partially supported by ONR grants N00014-14-1-0072 and N00014-17-1-2177.}\\
\texttt{goemans@math.mit.edu}
}

\maketitle

\begin{abstract}
We study the problem of community detection in a random hypergraph model which we call the stochastic block model for $k$-uniform hypergraphs ($k$-SBM). We investigate the exact recovery problem in $k$-SBM and show that a sharp phase transition occurs around a threshold: below the threshold it is impossible to recover the communities with non-vanishing probability, yet above the threshold there is an estimator which recovers the communities almost asymptotically surely. We also consider a simple, efficient algorithm for the exact recovery problem which is based on a semidefinite relaxation technique. 
\end{abstract}

\section{Introduction}

Identifying clusters from relational data is one of fundamental problems in computer science. It has many applications such as analyzing social networks \cite{newman2002random}, detecting protein-protein interactions \cite{marcotte1999detecting, chen2006detecting}, finding clusters in Hi-C genomic data \cite{cabreros2016detecting}, image segmentation \cite{shi2000normalized}, recommendation systems \cite{linden2003amazon,sahebi2011community} and many others. The goal is to find a community structure from relational measurements between data points.

Although many clustering problems are known to be NP-hard, typical data we encounter in applications are very different from the worst-case instances. This motivates us to study probabilistic models and average-case complexity for them. The stochastic block model (SBM) is one such model that has received much attention in the past few decades. In the SBM, we observe a random graph on the finite set of nodes where each pair of nodes is independently joined by an edge with probability only depending on the community membership of the endpoints. 

It is natural to consider the community detection problem for higher-order relations. A number of authors have already considered problems of learning from complex relational data \cite{agarwal2005beyond, govindu2005tensor, agarwal2006higher} and it has several applications such as folksonomy \cite{ghoshal2009random, zlatic2009hypergraph}, computer vision \cite{govindu2005tensor}, and network alignment problems for protein-protein interactions \cite{michoel2012alignment}. 

We consider a version of SBM for higher-order relations, which we call \emph{the stochastic block model for $k$-uniform hypergraph} ($k$-HSBM): we observe a random $k$-uniform hypergraph such that each set of nodes of size $k$ appears independently as an (hyper-)edge with probability only depending on the community labels of nodes in it. $k$-HSBM was first introduced in \cite{ghoshdastidar2014consistency} and investigated for its statistical limit in terms of detection \cite{lesieur2017statistical}, the minimax misclassification ratio \cite{lin2017fundamental,chien2018minimax}, and as a testbed for algorithms including naive spectral method \cite{ghoshdastidar2015provable,ghoshdastidar2015spectral,ghoshdastidar2017consistency}, spectral method along with local refinements \cite{abbe2016community,chien2018minimax,ahn2018hypergraph} and approximate-message passing algorithms \cite{angelini2015spectral,lesieur2017statistical}.

We focus on exact recovery, where our goal is to fully recover the community labels of the nodes from a random $k$-uniform hypergraph drawn from the model. For exact recovery, the maximum a posteriori (MAP) estimator always outperforms any other estimators in the sense that it has the highest probability of correctly recovering the solution. We prove that for the $k$-HSBM with two equal-sized and symmetric communities, exact recovery shows a sharp phase transition behavior, and moreover, the threshold can be characterized by the success of a certain type of local refinement. This type of phenomenon was mentioned as ``local-to-global amplification'' in \cite{abbe2016community}, and was proved in \cite{abbe2016exact} for the usual SBM with two symmetric communities (corresponds to $2$-HSBM) and more generally in \cite{abbe2015community} for SBMs with fixed number of communities. Our result can be regarded as a direct generalization of \cite{abbe2015community} to $k$-uniform hypergraphs.

Furthermore, we analyze a certain convex relaxation technique for the $k$-HSBM. 
We consider an algorithm which uses a semidefinite relaxation, based on the ``truncate-and-relax'' idea in our previous work \cite{kim2017community}. We prove that our algorithm guarantees exact recovery with high probability in a parameter regime which is orderwise optimal.

We remark that in \cite{abbe2016community} it was suggested that the local refinement methods together with an efficient partial recovery algorithm would imply the efficient exact recovery up to the information-theoretic threshold. An explicit algorithm exploiting this idea appears in \cite{chien2018minimax,ahn2018hypergraph} with a provable threshold for their algorithm to be successful. We note that the threshold of the algorithm of \cite{chien2018minimax} matches with the statistical threshold we derive, hence there is no gap between statistical and computational thresholds. On the other hand, we prove that our SDP-based algorithm does not achieve the statistical threshold when $k\geq 4$.

\subsection{The Stochastic Block Model for graphs: An overview}

Before we discuss the main topic of the paper, we start by discussing the usual stochastic block model to motivate our work.

The stochastic block model (SBM) has been one of the most fruitful research topics in community detection and clustering. One benefit of it is that, being a generative model we can formally study the probability of inferring the ground truth. While data from the real-world can behave differently, the SBM is believed to provide good insights in the field of community detection and has been studied for its sharp phase transition behavior \cite{mossel2013proof,abbe2015community,abbe2016exact}, computational vs. information-theoretic gaps \cite{chen2016statistical,abbe2015recovering}, and as a test bed for various algorithms such as spectral methods \cite{massoulie2014community,Vu2014}, semidefinite programs \cite{abbe2016exact,hajek2016achieving,javanmard2016phase}, belief-propagation methods \cite{decelle2011asymptotic,abbe2015detection,abbe2016achieving}, and approximate message-passing algorithms \cite{ver2014phase,cai2016inference,deshpande2016asymptotic,lesieur2017constrained}. We recommend \cite{abbe2016community} for a survey of this topic.

For the sake of exposition, let us consider the symmetric SBM with two equal-sized clusters, also known as the planted bisection model. Let $n$ be a positive integer, and let $p$ and $q$ be real numbers in $[0,1]$. The planted bisection model with parameter $n$, $p$ and $q$ is a generative model which outputs a random graph $G$ on $n$ vertices such that (i) the bipartition $(A,B)$ of $V$ defining two equal-sized clusters is chosen uniformly at random, and (ii) each pair $\{u,v\}$ in $V$ is connected independently with probability $p$ if $u$ and $v$ are in the same cluster, or probability $q$ otherwise. Note that this model coincides with Erd\H{o}s-R\'{e}nyi random graph model $\mathcal{G}(n,p)$, when $p$ and $q$ are equal. 

The goal is to find the ground truth $(A,B)$ either approximately or exactly, given a sampled graph $G$. We may ask the following questions regarding the quality of the solution.
\begin{itemize}
\item (Exact recovery) When can we find $(A,B)$ exactly (up to symmetry) with high probability?
\item (Almost exact recovery) Can we find a bipartition such that the vanishing portion of the vertices are mislabeled?
\item (Detection) Can we find a bipartition such that the portion of mislabeled vertices is less than $\frac{1}{2}-\epsilon$ for some positive constant $\epsilon$?
\end{itemize}
There are a number of works regarding these questions in the algorithmic point of view or in the sense of statistical achievability. The following is a short list of the states-of-the-art works regarding the model:
\begin{itemize}
\item Suppose that $p = \frac{a \log n}{n}$ and $q = \frac{b \log n}{n}$ where $a$ and $b$ are positive constants not depending on $n$. Then, exact recovery is possible if and only if $(\sqrt{a}-\sqrt{b})^2 > 2$. Moreover, there are efficient algorithms which achieves the information-theoretic threshold \cite{abbe2016exact,hajek2016achieving}. 
\item Suppose that $p = \frac{a}{n}$ and $q = \frac{b}{n}$ where $a$ and $b$ are positive constants not depending on $n$. Then, the detection is possible if and only if $\frac{(a-b)^2}{2(a+b)} > 1$. Moreover, there are efficient algorithms achieving the IT threshold \cite{mossel2012stochastic,mossel2013proof,massoulie2014community}.
\end{itemize} 
We further note that those sharp phase transition behaviours and algorithms achieving the threshold are found for general stochastic block models \cite{abbe2015recovering, abbe2016community}, \cite{lesieur2017statistical,lesieur2017constrained}. This paper focuses on exact recovery.

\subsection{The Stochastic Block Model for hypergraphs}

The stochastic block model for hypergraphs (HSBM) is a natural generalization of the SBM for graphs which was first introduced in \cite{ghoshdastidar2014consistency}. Informally, the HSBM can be thought as a generative model which returns a hypergraph with unknown clusters, and each hyperedge appears in the hypergraph independently with the probability depending on the community labels of the vertices involved in the hyperedge.

In \cite{ghoshdastidar2014consistency}, the authors consider the HSBM under the setting that the hypergraph generated by the model is $k$-uniform and dense. They consider a spectral algorithm on a version of hypergraph Laplacian, and prove that the algorithm exactly recovers the partition for $k > 3$ with probability $1-o_n(1)$. Subsequently, the same authors extended their results to sparse, non-uniform (but bounded order) setting, studying partial recovery \cite{ghoshdastidar2015spectral,ghoshdastidar2015provable,ghoshdastidar2017consistency}. 

We note that sparsity is an important factor to address in recovery problems of different types: exact recovery, almost exact recovery, and detection. In the case of the SBM for graphs, we recall that the average degree must be $\Omega(\log n)$ to assure exact recovery, and the average degree must be $\Omega(1)$ to assure detection. Conversely, the point of the sharp phase transition lies exactly in those regimes. We may expect similar behaviour for the $k$-uniform HSBM. For exact recovery, it was confirmed that the phase transition occurs in the regime of logarithmic average degree, by analyzing the optimal minimax risk of $k$-uniform HSBM \cite{lin2017fundamental,chien2018minimax}. For detection, the phase transition occurs in the regime of constant average degree \cite{angelini2015spectral}. The authors of \cite{angelini2015spectral} proposed a conjecture specifying the exact threshold point, based on the performance of belief-propagation algorithm. Also, such results for the weighted HSBM were independently proved in \cite{ahn2018hypergraph} and a exact threshold of the censored block model for uniform hypergraphs was classified in \cite{ahn2016community}.

In this paper, we consider a specific $k$-uniform HSBM with two equal-sized clusters. Let us remark that in the SBM case, we had two parameters $p$ and $q$ where the probability that an edge $\{i,j\}$ appears in the graph is $p$ or $q$ depending on whether $i$ and $j$ are in the same cluster or not. For an hyperedge of size greater than 2, there are different ways to generalize this notion, but we will focus on a simple model that the probability that a set $e$ of size $k$ appears as a hyperedge depends on whether $e$ is completely contained in a cluster or not. 

Let $n$ be a positive even number and let $V = [n]$ be the set of vertices of the hypergraph $\cH$. Let $k \geq 2$ be an integer. Let $p$ and $q$ be numbers between 0 and 1, possibly depending on $n$. We denote the collection of size $k$ subsets of $V$ by $\binom{V}{k}$. The $k$-HSBM with parameters $k$, $n$, $p$ and $q$, denoted $\HSBM(n,p,q;k)$, is a model which samples a $k$-uniform hypergraph $\cH$ on the vertex set $V$ according to following rules.
\begin{itemize}
\item $\sigma$ is a vector in $\{\pm 1\}^V$ chosen uniformly at random, among those with the equal number of $-1$'s and $1$'s. We may think $-1$ and $1$ as community labels.
\item Each $e = \{e_1,\cdots,e_k\}$ in $\binom{V}{k}$ appears independently as an hyperedge with probability
$$
\Pr(e \in E(\cH)) = \begin{cases} p &\text{if $\sigma_{e_1}=\sigma_{e_2}=\cdots=\sigma_{e_k}$} \\
q & \text{otherwise}. \end{cases}
$$
We say $e$ is \emph{in-cluster} with respect to $\sigma$ for the first case, and \emph{cross-cluster} w.r.t. $\sigma$ for the other case.
\end{itemize}

Our goal is to find the clusters from a given hypergraph $\cH$ generated from the model. We specially focus on exact recovery, formally defined as follows.

\begin{definition}
We say exact recovery in $\HSBM(n,p,q;k)$ is possible if there exists an estimator $\widehat{\sigma}$ which only fails to recover $\sigma$ up to a global sign flip with vanishing probability, i.e.,
$$
\Pr_{(\sigma,\cH) \sim \HSBM(n,p,q;k)}(\widehat{\sigma}(\cH) \not\in \{\sigma,-\sigma\}) = o_n(1).
$$
On the other hand, we say exact recovery in $\HSBM(n,p,q;k)$ is impossible if any estimator $\widehat{\sigma}$ fails to recover $\sigma$ up to a global sign flip with probability $1-o_n(1)$, i.e.,
$$
\Pr_{(\sigma,\cH) \sim \HSBM(n,p,q;k)}(\widehat{\sigma}(\cH) \not\in \{\sigma,-\sigma\}) = 1-o_n(1) \text{ for any $\widehat{\sigma}$}.
$$
\end{definition}

We remark that $\cH$ must be connected for exact recovery to be successful. In Erd\H{o}s-R\'{e}nyi (ER) model for random hypergraphs, it is known that a random hypergraph from the ER model is connected with high probability only if the expected average degree is at least $\frac{c(k-1)\log n}{\binom{n-1}{k-1}}$ for some $c > 1$\footnote{The proof for this result is a direct adaptation of the proof in \cite{bollobas1998random} for $k=2$, i.e., random graph model. See \cite{coja2007counting,behrisch2010order,cooley2015evolution} for phase transitions regarding giant components, which justifies the regime for partial recovery and detection.}. Together with the works in \cite{ahn2018hypergraph} and \cite{chien2018minimax}, this motivates us to work on the parameter regime where 
$$
p = \frac{\alpha\log n}{\binom{n-1}{k-1}} \quad \text{and} \quad q = \frac{\beta \log n}{\binom{n-1}{k-1}}
$$
for some constant $\alpha$ and $\beta$.

\subsection{Main results}

We first establish a sharp phase transition behaviour for exact recovery in the stochastic block model for $k$-uniform 
hypergraphs. We will assume that the parameter $k$ is a fixed positive integer not depending on $n$, and edge probabilities decay as
$$
p = \frac{\alpha \log n}{\binom{n-1}{k-1}} \quad \text{and} \quad q = \frac{\beta \log n}{\binom{n-1}{k-1}}
$$
where $\alpha$ and $\beta$ are fixed positive constants. Asymptotics in this paper are based on $n$ growing to infinity, unless noted otherwise. 

\begin{theorem}
\label{thm:A}
Exact recovery in $\HSBM(n,p,q;k)$ is possible if $I(\alpha,\beta)>1$, and impossible if $I(\alpha,\beta)<1$ where
$I(\alpha,\beta) = \frac{1}{2^{k-1}}(\sqrt{\alpha}-\sqrt{\beta})^2$. 
\end{theorem}

In case of exact recovery, the maximum a posteriori (MAP) estimator achieves the minimum error probability. The MAP estimator corresponds to the maximum-likelihood (ML) estimator in this model since the partition is chosen from a uniform distribution. Hence, it is sufficent to analyze the performance of the ML estimator to prove Theorem \ref{thm:A}.   

On the other hand, we ask whether there exists an efficient algorithm which recover the hidden partition $\sigma$ achieving the information-theoretic threshold. Note that the ML estimator (which achieves the minimum error probability) is given by
$$
\widehat{\sigma}_{MLE}(H) = \argmax_{x \in \{\pm 1\}^V: \one^T x = 0} \Pr_{(\sigma,\cH) \sim \HSBM(n,p,q;k)}(\cH=H|\sigma=x).
$$
This is in general hard to compute. For example, when $k=2$ and $p > q$, it reduces to find a balanced bipartition with the minimum number of edges crossing given a graph $G$, also known as MIN-BISECTION problem which is NP-hard. However, there is a simple and efficient algorithm which works up to the threshold of the ML estimator in case of $k=2$. This algorithm is based on a standard semidefinite relaxation of MIN-BISECTION \cite{goemans1995improved}.

For general $k$-HSBM, we propose an efficient algorithm using a ``truncate-and-relax'' strategy. Given a $k$-uniform hypergraph $H$ on the vertex set $V$, let us define a weighted graph $(G_H, w)$ on the same vertex set where the weights are given by
$$
w(ij) = \#(e \in E(H): \{i,j\} \subseteq e)
$$
for each $\{i,j\} \in \binom{V}{2}$. Let $\widehat{\sigma}_{trunc}$ be an optimal solution of 
$$
\text{maximize} \sum_{ij \in \binom{V}{2}} w(ij)x_i x_j \quad \text{subject to} \quad x \in \{\pm 1\}^V \text{ and } \one^T x = 0,
$$
which is equivalent to finding the min-bisection of the weighted graph $(G_H, w)$. Now, consider the following semidefinite program:
\begin{equation}
\label{eqn:sdp}
\begin{aligned}
\text{maximize} & & & \sum_{ij \in \binom{V}{2}} w(ij) X_{ij} \\
\text{subject to} & & & \sum_{i,j \in V} X_{ij} = 0, \\
& & & X_{ii} = 1 \text{ for all $i \in V$,} \\
& & & X=X^T \succeq 0.
\end{aligned}
\end{equation}
This program is a relaxation of the min-bisection problem above, since for any feasible $x$ in the original problem corresponds to a feasible solution $X = xx^T$ in the relaxed problem.

The ML estimator attempts to maximize the function
$$
f_H(x) = \log \Pr_{(\sigma,\cH)\sim \HSBM(n,p,q;k)}(\cH = H| \sigma=x)
$$
over the vectors in the hypercube $\{\pm 1\}^V$ with equal number of $-1$'s and $1$'s. We can write $f_H(x)$ as a multilinear polynomial in $x$, since $x_i^2 = 1$ for all $i$. Let $f_H^{(2)}(x)$ be the quadratic part of $f_H(x)$. Then,
maximizing $f_H^{(2)}(x)$ is equivalent to find the min-bisection of $(G_H, w)$. This justifies our term truncate-and-relax, as in our previous work \cite{kim2017community}.

Now, let $\widehat{\Sigma}(H)$ be the solution of (\ref{eqn:sdp}). We prove that this estimator correctly recovers the hidden partition with high probability up to a threshold which is order-wise optimal.

\begin{theorem}
\label{thm:B}
Suppose $\alpha > \beta$. Then $\widehat{\Sigma}(\cH)$ is equal to $\sigma \sigma^T$ with probability $1-o_n(1)$ if $I_{sdp}(\alpha,\beta)>1$ where
$$
I_{sdp}(\alpha,\beta) = \frac{k-1}{2^{2k}} \cdot \frac{(\alpha-\beta)^2}{\left(\frac{k}{2^{k}} \alpha+\left(1-\frac{k}{2^k}\right)\beta\right)}.
$$
\end{theorem}

It is natural to ask whether this analysis is tight. The proof proceeds by constructing a dual solution which certifies that $\sigma\sigma^T$ is the unique optimum of (\ref{eqn:sdp}) with high probability. Following \cite{Bandeira2016laplacian}, the dual solution (if exists) is completely determined by $(G_H, w)$ which has the form of a ``Laplacian'' matrix. Precisely, the major part of the proof is devoted to prove that the matrix $L$ of size $V \times V$ with entries
$$
L_{ij} = \begin{cases}
-w(ij)\sigma_i \sigma_j & \text{if $i \neq j$} \\
\sum_{i' \in V\setminus \{i\}} w(ii')\sigma_i\sigma_{i'} & \text{if $i=j$},
\end{cases}
$$
is positive-semidefinite with high probability. We use the Matrix Bernstein inequality to prove that the fluctuation $\|L - \E L\|$ is smaller compared to the minimum eigenvalue of $\E L$ w.h.p., under the assumption $I_{sdp}(\alpha,\beta) > 1$. However, we believe that it can be improved by a direct analysis of $\|L - \E L\|$. Numerical simulations and discussions which supports our belief can be found in Section \ref{sec:conjecture}.

Finally, we complement Theorem \ref{thm:B} by providing a lower bound of the truncate-and-relax algorithm. Recall that the algorithm tries to find a solution in the relaxed problem (\ref{eqn:sdp}). It implies that if the min-bisection of $(G_H, w)$ is not the correct partition $\sigma$, then the truncate-and-relax algorithm will also return a solution which is not equal to $\sigma\sigma^T$. Hence, we have
$$
\Pr(\widehat{\Sigma}(\cH) \neq \sigma\sigma^T) \geq \Pr(\widehat{\sigma}_{trunc}(\cH) \not\in \{\sigma,-\sigma\}).
$$
We find a sharp threshold for the estimator $\widehat{\sigma}_{trunc}(\cH)$ recovering $\sigma$ or $-\sigma$ successfully.

\begin{theorem}
\label{thm:C}
Suppose $\alpha > \beta$. Let $I_2(\alpha,\beta)$ be defined as following:
$$
I_2(\alpha,\beta) = \max_{t\geq 0} 
\frac{1}{2^{k-1}}\left[ \alpha (1-e^{-(k-1)t}) + \sum_{a=1}^{k-1} \beta \binom{k-1}{a} \left(1-e^{-(k-1-2a)t}\right)\right]
$$
If $I_2(\alpha,\beta) < 1$, then $\widehat{\sigma}_{trunc}(\cH)$ is not equal to neither $\sigma$ nor $-\sigma$ with probability $1-o_n(1)$. On the other hand, if $I_2(\alpha,\beta) > 1$, then $\widehat{\sigma}_{trunc}(\cH)$ is either of $\sigma$ or $-\sigma$ with probability $1-o_n(1)$.
\end{theorem}

\begin{figure}
\centering
\includegraphics[scale=0.6]{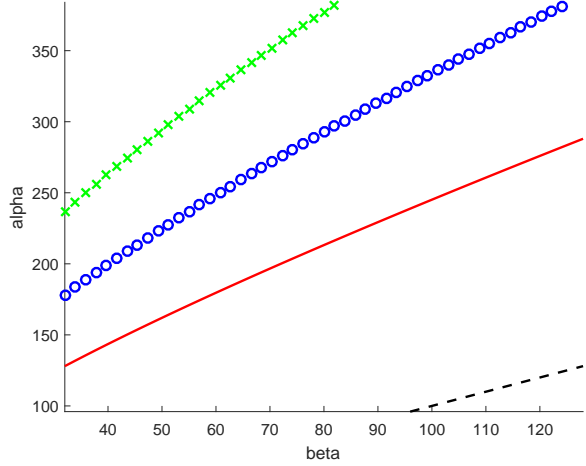}
\caption{Visualization of $I$, $I_2$, $I_{sdp}$ when $k=6$: (a) the solid line represents $I(\alpha,\beta)=1$, (b) the circled line represents $I_2(\alpha,\beta)=1$, and (c) the x-marked line represents $I_{sdp}(\alpha,\beta)=1$. The dashed black line is the graph of $\alpha = \beta$.}
\label{fig:thresholds}
\end{figure}

We note that
$$
I(\alpha,\beta) = \max_{t \geq 0} \frac{1}{2^{k-1}}\left(\alpha(1-e^{-(k-1)t}) + \beta(1-e^{(k-1)t})\right),
$$
hence $I(\alpha,\beta) < I_2(\alpha,\beta)$ for any $\alpha > \beta > 0$. Figure \ref{fig:thresholds} shows the relations between $I$, $I_2$ and $I_{sdp}$ for $k=6$.

Theorem \ref{thm:C} and the discussion above implies that the truncate-and-relax algorithm fails with probability $1-o_n(1)$ if $I_2(\alpha,\beta) > 1$. We conjecture that this is the correct threshold of the performance of the algorithm. In future work, we will attempt to prove this conjecture by improving the matrix concentration bound as discussed above.

\begin{conjecture}
If $I_2(\alpha,\beta) > 1$, then $\widehat{\Sigma}(\cH) = \sigma\sigma^T$ with probability $1-o_n(1)$. 
\end{conjecture}

\section{Maximum-likelihood estimator}

Recall that $\widehat{\sigma}_{MLE}(H)$ is a maximizer of the likelihood probability $\Pr(\cH=H|\sigma=x)$ (ties are broken arbitrarily). Let $f_H(x) = \log \Pr(\cH=H|\sigma=x)$ for $x \in \{\pm 1\}^V$.

For brevity, let us first introduce a few notations. Let $x \in \{\pm 1\}^V$. Let $x^{\oeq k}$ be a vector in $\{0,1\}^{\binom{V}{k}}$ where
$$
(x^{\oeq k})_e = \begin{cases}
1 & \text{if $x_{e_1}=x_{e_2}=\cdots=x_{e_k}$} \\
0 & \text{otherwise}
\end{cases}
$$
for each $e = \{e_1,\cdots,e_k\} \subseteq V$. 
Let $H$ be a $k$-uniform hypergraph on the vertex set $V$ with the edge set $E(H)$. Let $A_H$ be the vector in $\{0,1\}^{\binom{V}{k}}$ such that
$$
(A_H)_e = \begin{cases}
1 & \text{if $e \in E(H)$} \\
0 & \text{otherwise}
\end{cases}
$$
for each $e \in \binom{V}{k}$. Note that 
\begin{eqnarray*}
\inprod{A_H,x^{\oeq k}}&=&\sum_{e \in \binom{V}{k}} (A_H)_e (x^{\oeq k})_e \\
&=& \sum_{e \in E(H)} \one\{\text{$e$ is in-cluster with respect to $x$}\}.
\end{eqnarray*}
Hence, $\inprod{A_H,x^{\oeq k}}$ is equal to the number of in-cluster edges in $H$ with respect to the partition $x$. 

The ML estimator tries to find the ``best'' partition $x \in \{\pm 1\}^V$ with equal number of $1$'s and $-1$'s. Intuitively, if $p > q$, i.e., in-cluster edges appears more likely than cross-cluster edges (we call such case \emph{assortative}), then the best partition will correspond to $x$ which maximizes the number of in-cluster edges w.r.t. $x$. On the other hand, if $p < q$ (we call such case \emph{disassortative}) then the best partition will corresponds to the minimizer, respectively. The following proposition confirms this intuition. We defer the proof to Section \ref{sec:propMLE} in the appendix.

\begin{proposition}
\label{prop:MLE}
The ML estimator $\widehat{\sigma}_{MLE}(H)$ is the maximizer (minimizer, respectively) of $\inprod{A_H,x^{\oeq k}}$ if $p > q$ (if $p < q$, respectively) over all $x \in \{\pm 1\}^V$ such that $\one^T x = 0$. 
\end{proposition}

\section{Sharp phase transition in $\HSBM(n,p,q;k)$}
\label{sec:IT}

In this section, we prove Theorem \ref{thm:A}. The techniques we use can be seen as a hypergraph extension of the techniques used in \cite{abbe2016exact}. 

Informally, we are going to argue that the event for the ground truth $\sigma$ being the best guess (i.e. $\sigma$ is the global optimum of the likelihood function) can be approximately decomposed into the events that $\sigma$ is unimprovable by flipping the label of $v$ for $v \in V$. This type of phenomenon was called \emph{local-to-global amplification} in \cite{abbe2016community} which seems to hold for more general classes of the graphical model. 

Let $p_{fail}$ be the probability that the ML estimator fails to recover the hidden partition, i.e.,
$$
p_{fail} = \Pr_{(\sigma,\cH)\sim \HSBM(n,p,q;k)} \left(\widehat{\sigma}_{MLE}(\cH) \not\in \{\sigma,-\sigma\} \right).
$$
As we have seen in the previous section, the ML estimator $\widehat{\sigma}_{MLE}(H)$ is a maximizer of $\inprod{A_H, x^{\oeq k}}$ over the choices of $x \in \{\pm 1\}^V$ such that $\one^T x = 0$. Thus, $p_{fail}$ is equal to the probability that $\inprod{A_\cH,x^{\oeq k}} \geq \inprod{A_\cH, \sigma^{\oeq k}}$ happens for some $x \not\in\{\sigma,-\sigma\}$ satisfying the balance condition $\one^T x= 0$.

\subsection{Lower bound}

We first prove the impossibility part of Theorem \ref{thm:A}. For concreteness, we focus on the assortative case, i.e., $p > q$ but the proof can be easily adapted for the disassortative case. 

Before we prove the lower bound, let us consider the usual stochastic block model for graphs which corresponds to $k=2$ in order to explain the intuition of the proof. Given a sample $G$, partition $\sigma$ and a vertex $v \in V$, let us define the \emph{in-degree} of $v$ as
$$
\mathrm{indeg}_{G,\sigma}(v) = \#(vw \in E(G): \sigma_v = \sigma_w),
$$
and the \emph{out-degree} of $v$ as
$$
\mathrm{outdeg}_{G,\sigma}(v;\sigma) = \#(vw \in E(G): \sigma_v \neq \sigma_w).
$$
We will omit the subscript $G,\sigma$ if the context is clear.

Suppose that there are vertices $v$ and $w$ from different clusters such that the in-degree of each vertex is smaller than the out-degree of each vertex. In this case, swapping the label of $v$ and $w$ will yield a new balanced partition with greater number of in-cluster edges, hence the ML estimator will fail to recover $\sigma$. Now, suppose that
$$
\Pr(\mathrm{indeg}(v) < \mathrm{outdeg}(v)) = \omega(n^{-1})
$$
for all $v$. If those events were independent, we would get
\begin{eqnarray*}
1-\Pr(\exists v \text{ s.t. } \sigma_v=1, \mathrm{indeg}(v) < \mathrm{outdeg}(v)) &=& \prod_{v:\sigma_v=1}\left(1-\Pr(\mathrm{indeg}(v) < \mathrm{outdeg}(v))\right) \\
&\leq& (1-\omega(n^{-1}))^{n/2} \leq e^{-\omega(1)},
\end{eqnarray*}
and vise versa for $w$. It would imply that there is a ``bad'' pair $(v,w)$ with probability $1-o_n(1)$ hence the ML estimator fails with probability $1-o_n(1)$. We remark that this argument is not mathematically because the in-degrees of vertex $v$ and $w$ (as well as out-degrees of them) are \emph{not} independent as they share a variable indicating whether $\{v,w\}$ is an edge or not. However, we can overcome it by conditioning on highly probable event which makes those events independent, as in \cite{abbe2016exact}.

We extend the definitions of in-degree and out-degree for the $k$-HSBM as
\begin{eqnarray*}
\mathrm{indeg}_{H,\sigma}(v) &=& \#(e \in E(H): v \in e, \text{ $e$ is in-cluster w.r.t. $\sigma$}), \\
\mathrm{outdeg}_{H,\sigma}(v) &=& \#(e \in E(H): v \in e, \text{ $e$ is cross-cluster but $e\setminus\{v\}$ is in-cluster w.r.t. $\sigma$}).
\end{eqnarray*}
Observe that they coincide with the corresponding definition for the usual SBM ($k=2$). We note that the sum of in-degree and out-degree is not equal to the degree of $v$, the number of hyperedges in $H$ containing $v$ when $k \geq 3$. We extended those definitions in this way because any edge $e$ which is neither in-cluster nor cross-cluster but $e\setminus\{v\}$ is in-cluster does not contribute on $\inprod{A_\cH, x^{\oeq k}}$ when we flip the sign of the label of $v$.

Now, note that the in-degree and the out-degree of $v$ are independent binomial random variables with different parameters. To estimate the probability 
$$
\Pr\left(\mathrm{indeg}(v)-\mathrm{outdeg}(v) < 0\right),
$$
we provide a tight estimate for the tail probability of a weighted sum of independent binomial variables in Section \ref{sec:tail}. Precisely we prove that
$$
\Pr(\mathrm{indeg}(v) - \mathrm{outdeg}(v) < -\delta \log n) = n^{-I(\alpha,\beta)+o_n(1)}
$$ 
as long as $\delta=\delta(n)$ vanishes as $n$ grows, where
$$
I(\alpha,\beta) = \frac{1}{2^{k-1}}\left(\sqrt{\alpha}-\sqrt{\beta}\right)^2.
$$
As we discussed, if $I(\alpha,\beta)<1$ then the tail probability is of order $\omega(n^{-1})$ and it implies that the ML estimator fails with probability $1-o_n(1)$. 

\begin{theorem}
Let $I(\alpha,\beta) = \frac{1}{2^{k-1}}(\sqrt{\alpha}-\sqrt{\beta})^2$. If $I(\alpha,\beta)<1$, then $p_{fail} = 1-o_n(1)$. 
\end{theorem}

\begin{proof}

Let $A = \{v \in V: \sigma_v = +1\}$ and $B = V \setminus A$. For $a \in A$ and $b \in B$, let us define $\sigma^{(ab)}$ to be the vector obtained by flipping the signs of $\sigma_a$ and $\sigma_b$. By definition, $\sigma^{(ab)}$ is balanced. We are going to prove that with high probability there exist $a \in A$ and $b \in B$ such that $\inprod{A_\cH, \sigma^{\oeq k}} \leq \inprod{A_\cH, (\sigma^{(ab)})^{\oeq k}}$. For simplicity, let $\Sigma = \sigma^{\oeq k}$ and $\Sigma^{(ab)} = (\sigma^{(ab)})^{\oeq k}$.

Note that 
\begin{eqnarray*}
\inprod{A_\cH, \Sigma} - \inprod{A_\cH, \Sigma^{(ab)}}
&=& \left(\mathrm{indeg}_{\cH,\sigma}(a)-\mathrm{outdeg}_{\cH,\sigma}(a)\right)\\
& & +\left(\mathrm{indeg}_{\cH,\sigma}(b)-\mathrm{outdeg}_{\cH,\sigma}(b)\right).
\end{eqnarray*}
For $v \in V$, let $E_v$ be the event such that 
$$
\mathrm{indeg}_{\cH,\sigma}(v) - \mathrm{outdeg}_{\cH,\sigma}(v) \leq 0
$$
holds. Then, $E_a \cap E_b$ implies that $\inprod{A_\cH, \Sigma} - \inprod{A_\cH, \Sigma^{(ab)}} \leq 0$. Hence
$$
p_{fail} = \Pr\left(\exists a\in A, b\in B: \inprod{A_\cH, \Sigma} - \inprod{A_\cH, \Sigma^{(ab)}}\right)
\geq \Pr\left(\bigcup_{a\in A} E_a \cap \bigcup_{b\in B} E_b\right).
$$
We recall that if $E_v$ for $v \in V$ were mutually independent, we can exactly express the right-hand side as
$$
\left(1-\prod_{a\in A} \Pr(\neg E_a)\right)\left(1-\prod_{b\in B} \Pr(\neg E_b)\right)
$$
but unfortunately it is not the case. To see this, let us fix $a \in A$ and $a' \in A$. Then, we have
\begin{eqnarray*}
\mathrm{indeg}(a) - \mathrm{outdeg}(a) &=& \sum_{e \ni a: e\subseteq A} (A_\cH)_e - \sum_{e \ni a: e\cap A = \{a\}} (A_\cH)_e, \text{ and}\\
\mathrm{indeg}(a') - \mathrm{outdeg}(a') &=& \sum_{e \ni a': e\subseteq A} (A_\cH)_e - \sum_{e \ni a': e\cap A = \{a'\}} (A_\cH)_e.
\end{eqnarray*}
They share variables $(A_\cH)_e$ for $e$ satisfying $\{a,a'\}\subseteq e \subseteq A$. The expected contribution of those variables is $p\binom{|A|-2}{k-2} = o_n(1)$, so we may expect
$$
\Pr(E_a \cup E_{a'}) \approx 1-\Pr(\neg E_a)\Pr(\neg E_{a'}).
$$
In the similar spirit, we are going to prove that for an appropriate choice of $U \subseteq V$, the events $\{E_a\}_{a \in U \cap A}$ and $\{E_b\}_{b \in U \cap B}$ are approximately independent, so
\begin{eqnarray*}
p_{fail} &\geq& \Pr\left(\bigcup_{a\in A\cap U} E_a \cap \bigcup_{b\in B\cap U} E_b\right) \\
&\approx& \left(1-\prod_{a \in A\cap U} \Pr(\neg E_a)\right)\left(1-\prod_{b \in B\cap U} \Pr(\neg E_b)\right).
\end{eqnarray*}
Together with the tight estimate on $\Pr(E_v)$, it would give us a good lower bound on $p_{fail}$.  

Let $U \subseteq V$ be a set of size $\gamma n$ where $|U \cap A| = |U\cap B|$. We will choose $\gamma = \gamma(n)$ later to be poly-logarithmically decaying function in $n$. Let $\cS$ be the set of $e \in \binom{V}{k}$ such that $e$ contains at least two vertices in $U$. We would like to condition on the values of $\{(A_\cH)_e\}_{e\in\cS}$, which captures all dependency occurring among $E_v$'s for $v \in U$. 

Let $\delta = \delta(n)$ be a positive number depending on $n$ which we will choose later, and let $F$ be the event that the inequality
$$
\max_{v \in U} \sum_{e \in \cS: e \ni v} (A_\cH)_e \leq \delta \log n
$$
holds. For each $a \in A \cap U$, let $E_a'$ be the event that the inequality
$$
\sum_{e\subseteq A: e \cap U = \{a\}} (A_\cH)_e - \sum_{\substack{e: e \cap U = \{a\} \\ e\setminus \{a\} \subseteq B}} (A_\cH)_e \leq -\delta \log n
$$
is satisfied. We claim that $E_a' \cap F \subseteq E_a$. It follows from the direct calculation, as if we assume $E_a' \cap F$, then
\begin{eqnarray*}
\mathrm{indeg}(v)-\mathrm{outdeg}(v)
&=& \sum_{e \ni a: e\subseteq A} (A_\cH)_e - \sum_{e \ni a: e\cap A = \{a\}} (A_\cH)_e \\
&\leq& \sum_{e\subseteq A: e \cap U = \{a\}} (A_\cH)_e - \sum_{\substack{e: e \cap U = \{a\} \\ e\setminus \{a\} \subseteq B}} (A_\cH)_e + \sum_{e \in \cS: e\ni a} (A_\cH)_e \\
&\leq& \sum_{e\subseteq A: e \cap U = \{a\}} (A_\cH)_e - \sum_{\substack{e: e \cap U = \{a\} \\ e\setminus \{a\} \subseteq B}} (A_\cH)_e + \delta \log n \\
&\leq& 0.
\end{eqnarray*}
We get
$$
p_{fail} \geq \Pr\left(\bigcup_{a\in A\cap U} E_a \cap \bigcup_{b \in B\cap U} E_b\right) \geq
\Pr\left(\bigcup_{a\in A\cap U} E_a' \cap \bigcup_{b \in B\cap U} E_b' \middle| F\right) \Pr(F).
$$
Note that $E_v'$ only depends on the set of variables $\{(A_\cH)_e: e\cap U = \{v\}\}$, which are mutually disjoint for $v \in U$. Also, $\{(A_\cH)_e: e \in \cS\}$ is disjoint with any of those sets of variables. Hence, events $F$ and $\{E_v'\}_{v\in U}$ are mutually independent, and we get
\begin{eqnarray*}
\Pr\left(\bigcup_{a\in A\cap U} E_a' \cap \bigcup_{b \in B\cap U} E_b' \middle| F\right) &=& 
\Pr\left(\bigcup_{a\in A\cap U} E_a'\right)\Pr\left(\bigcup_{b\in B\cap U} E_b'\right) \\
&=& \left(1-\prod_{a\in A\cap U} \Pr(\neg E_{a}')\right)\left(1-\prod_{b\in B\cap U} \Pr(\neg E_{b}')\right)
\end{eqnarray*}

We claim that 
$$
\prod_{a\in A\cap U} \Pr(\neg E_{a}') = o_n(1), \quad
\prod_{b\in B\cap U} \Pr(\neg E_{b}') = o_n(1) \quad \text{and} \quad 
\Pr(\neg F) = o_n(1),
$$
for appropriate choice of $\gamma$ and $\delta$. This immediately implies that $p_{fail}=1-o_n(1)$ as desired. 

Let us first prove $\Pr(\neg F) = o_n(1)$. Let $X_v$ be the random variable defined as
$$
X_v := \sum_{e \in \cS: e\ni v} (A_\cH)_e,
$$
for $v \in U$. We have
$$
\Pr(\neg F) = \Pr(\exists v\in U: X_v > \delta \log n) \leq \sum_{v \in U} \Pr(X_v > \delta \log n)
$$
by a union bound. Note that
\begin{eqnarray*}
\E X_v &\leq& \max(p,q) \#(e: e \ni v, |e\cap U| \geq 2) \\
&=& \frac{\max(\alpha,\beta)\log n}{\binom{n-1}{k-1}} \left(\binom{n-1}{k-1} - \binom{n-|U|}{k-1}\right) \\
&\lesssim& \left(1-(1-\gamma)^{k-1}\right)\log n = \Theta(\gamma\log n).
\end{eqnarray*}
Using a standard Chernoff bound, we get the following lemma. For completeness, we include the proof in the appendix (Section \ref{sec:tailbound0}). 

\begin{lemma}
\label{lemma:tailbound0}
Let $X$ be a sum of independent Bernoulli variables such that $\E X = \Theta(\gamma \log n)$ where $\gamma = o_n(\log^{-1} n)$. Let $\delta$ be a positive number which decays to 0 as $n$ grows, with $\delta = \omega_n(\log^{-1} n)$.  Then,
$$
\Pr\left(X > \delta \log n \right) \leq n^{-\delta \log \frac{\delta}{\gamma} + o_n(1)}.
$$
\end{lemma}

Letting $\gamma = \log^{-3} n$ and $\delta = (\log \log n)^{-1}$, we get 
$$
\delta \log \frac{\delta}{\gamma} = \frac{3\log \log n - \log \log \log n}{\log \log n} = 3-o_n(1)
$$
and so $\Pr(\neg F) = n^{-3+o_n(1)} = o_n(1)$. 

Now, we would like to prove that 
$$
\prod_{a\in A\cap U} \Pr(\neg E_{a}') = o_n(1) \quad \text{and} \quad \prod_{b\in B\cap U} \Pr(\neg E_{b}') = o_n(1)
$$
by showing that 
$$
\Pr(E_v') \geq n^{-I(\alpha,\beta)+o_n(1)}\quad \text{where } I(\alpha,\beta) = \frac{1}{2^{k-1}}(\sqrt{\alpha}-\sqrt{\beta})^2
$$
for any $v \in U$. This implies that
\begin{eqnarray*}
\prod_{a\in A\cap U} \Pr(\neg E_{a}') &\leq& \left(1-n^{-I(\alpha,\beta)+o_n(1)}\right)^{\gamma n/2} \\
&\leq& \exp\left(-\frac{\gamma}{2} n^{1-I(\alpha,\beta)+o_n(1)}\right),
\end{eqnarray*}
and since we assumed that $I(\alpha,\beta)<1$ and $\gamma = \log^{-3} n$, we get
$$
\prod_{a\in A\cap U} \Pr(\neg E_{a}') \leq e^{-\frac{n^{1-I(\alpha,\beta)}}{\log^3 n}} = o_n(1),
$$
and similarly $\prod_{b\in B\cap U} \Pr(\neg E_{b}') = o_n(1)$ as desired.

To estimate the probability that $E_a'$ happens, let $Y_a$ and $Z_a$ be random variables defined as
$$
Y_a = \sum_{e\subseteq A: e \cap U = \{a\}} (A_\cH)_e \quad \text{and} \quad Z_a=\sum_{\substack{e: e \cap U = \{a\} \\ e\setminus \{a\} \subseteq B}} (A_\cH)_e.
$$
Recall that $E_a'$ is the event that $Y_a-Z_a \leq -\delta \log n$ holds. 

\begin{lemma}
\label{lemma:tailbound1}
Let $Y$ be a binomial random variable from $\mathrm{Bin}(N, p)$ and $Z$ be a binomial random variable from $\mathrm{Bin}(N, q)$ where $N = \left(\frac{1}{2^{k-1}}\pm o_n(1)\right)\binom{n-1}{k-1}$, $p = \frac{\alpha \log n}{\binom{n-1}{k-1}}$ and $q = \frac{\beta \log n}{\binom{n-1}{k-1}}$. Let $I(\alpha,\beta) = \frac{1}{2^{k-1}}(\sqrt{\alpha}-\sqrt{\beta})^2$ and let $\delta$ be a positive number vanishing as $n$ grows. Then,
$$
\Pr(Y - Z \leq -\delta \log n) = n^{-I(\alpha,\beta)+o_n(1)}.
$$
\end{lemma}

In fact, we derive a generic tail bound for weighted sum of binomial random variables (Theorem \ref{thm:tail}) in Section \ref{sec:tail} of the appendix. Lemma \ref{lemma:tailbound1} is a direct corollary of Theorem \ref{thm:tail} and we defer the proof to Section \ref{sec:tailbound1}.
\end{proof}

\subsection{Upper bound}

We are going to use a union bound to prove the upper bound. Let $x$ and $\sigma$ be vectors in $\{-1,+1\}^V$. The \emph{Hamming distance} between $x$ and $\sigma$ (denoted $d(x,\sigma)$) is defined as the number of $v \in V$ such that $x_v \neq \sigma_v$. Note that if $x$ and $\sigma$ are balanced, then
\begin{eqnarray*}
d(x,\sigma) &=& \#(v \in V: x_v = 1, \sigma_v = -1) + \#(v \in V: x_v = -1, \sigma_v = 1) \\
&=& \#(v \in V: x_v = 1) + \#(v \in V: \sigma_v=1)-2\#(v \in V: x_v=\sigma_v=1) \\
&=& n - 2\#(v \in V: x_v=\sigma_v=1),
\end{eqnarray*}
hence $d(x,\sigma)$ is even. 

Let us fix $\sigma$ and let $\cH$ be a $k$-uniform random hypergraph generated by the model under the ground truth $\sigma$. We note that the distribution of the random variable $\inprod{A_\cH, x^{\oeq k}-\sigma^{\oeq k}}$ is invariant under the permutation of $V$ preserving $\sigma$, hence it only depends on $d(x, \sigma)$. Hence, there is a quantity $p_{fail}^{(d)}$ which satisfies
$$
p_{fail}^{(d)} = \Pr(\inprod{A_\cH,x^{\oeq k} - \sigma^{\oeq k}} \geq 0)
$$
for any $x$ with $d(x,\sigma)=d$. Moreover, $p_{fail}^{(d)} = p_{fail}^{(n-d)}$ since our model is invariant under a global sign flip.

Recall that the ML estimator fails to recover $\sigma$ if and only if
$$
\inprod{A_\cH,x^{\oeq k}} \geq \inprod{A_\cH,\sigma^{\oeq k}}
$$
for some balanced $x \in \{\pm 1\}^V$ which is neither $\sigma$ nor $-\sigma$.
We remark that we count the equality as a failure, which will only make $p_{fail}$ larger. By union bound, we have
\begin{eqnarray*}
p_{fail} &\leq& \sum_{\substack{x \in \{\pm 1\}^V: \one^T x = 0, \\ 1\leq d(x,\sigma)\leq n-1}} \Pr(\inprod{A_\cH,x^{\oeq k} - \sigma^{\oeq k}} \geq 0) \\
&\leq& 2\sum_{\substack{d:1 \leq d \leq \frac{n}{2} \\ d\text{ is even}}} p_{fail}^{(d)} \cdot \#(x \in \{\pm 1\}^V: \one^T x = 0, d(x,\sigma)=d).
\end{eqnarray*}
We note that there is a one-to-one correspondence between a balanced $x$ and a pair of sets $(U_+,U_-)$ where
\begin{eqnarray*}
U_+ &=& \{v: x_v=-1, \sigma_v=1\} \subseteq \{v: \sigma_v=+1\} \\
U_- &=& \{v: x_v=1, \sigma_v=-1\} \subseteq \{v: \sigma_v=-1\},
\end{eqnarray*}
and we must have $d(x,\sigma)=2|U_+|=2|U_-|$ since $x$ is balanced. Hence, the number of balanced $x$'s with $d(x,\sigma)=d$ is equal to $\binom{n/2}{d/2}^2$. We have
$$
p_{fail} \leq 2\sum_{\substack{d:1 \leq d \leq \frac{n}{2} \\ d\text{ is even}}} \binom{n/2}{d/2}^2 p_{fail}^{(d)}.
$$
Now, let us formally state the main result of this section.

\begin{theorem}
Suppose that $I(\alpha,\beta) > 1$. Then,
$$
p_{fail} \leq n^{-\frac{I(\alpha,\beta)-1}{2} + o_n(1)}
$$
and it implies that $p_{fail} = o_n(1)$.
\end{theorem}

\begin{proof}
Let $d$ be even number in between 1 and $\frac{n}{2}$. Choose any balanced $x$ with $d(\sigma,x)=d$, and let $X_d$ be 
$$
X_d := \inprod{A_\cH, x^{\oeq k}-\sigma^{\oeq k}} = \sum_{e \in \binom{V}{k}} (A_\cH)_e \left(x^{\oeq k}-\sigma^{\oeq k}\right)_e.
$$
Let $A = \{v: \sigma_v = 1\}$ and $A' = \{v: x_v = 1\}$. We say $e$ crosses $A$ if $e \cap A$ and $e \setminus A$ are both non-empty (and respectively for $A'$). Then,
$$
\left(x^{\oeq k}-\sigma^{\oeq k}\right)_e = 
\begin{cases}
-1 & \text{if $e$ doesn't cross $A'$ but crosses $A$} \\
1 & \text{if $e$ crosses $A'$ but doesn't cross $A$} \\
0 & \text{otherwise.}
\end{cases}
$$
Hence $X_d = Y_d - Z_d$ where $Y_d$ and $Z_d$ are binomial variables with $Y_d \sim \mathrm{Bin}(N_1,p)$ and $Z_d \sim \mathrm{Bin}(N_2,q)$ where
\begin{eqnarray*}
N_1 &=& \#(e\in \binom{V}{k}: \text{$e$ doesn't cross $A'$ but crosses $A$})\\
N_2 &=& \#(e\in \binom{V}{k}: \text{$e$ crosses $A'$ but doesn't cross $A$}).
\end{eqnarray*}
A simple combinatorial argument shows that 
\begin{eqnarray*}
N_1 &=& \#(e: \text{$e \subseteq A'$ and $e$ crosses $A$}) + \#(e: \text{$e \subseteq V\setminus A'$ and $e$ crosses $A$}) \\
&=& \left(\binom{|A'|}{k}-\binom{|A\cap A'|}{k} - \binom{|A'\setminus A|}{k}\right) \\
& & + \left(\binom{n-|A'|}{k}- \binom{|A\setminus A'|}{k}-\binom{n-|A\cup A'|}{k}\right) \\
&=& 2\left(\binom{n/2}{k}-\binom{d/2}{k}-\binom{(n-d)/2}{k}\right),
\end{eqnarray*}
and $N_2 = N_1$ by the symmetry. Hence,
$$
p_{fail}^{(d)} = \Pr_{\substack{Y_{d} \sim \mathrm{Bin}(N,q) \\ Z_{d} \sim \mathrm{Bin}(N,p)}}(Y_d - Z_d \geq 0)
$$
where $N = 2\left(\binom{n/2}{k} - \binom{d/2}{k} - \binom{(n-d)/2}{k}\right)$, $p = \frac{\alpha \log n}{\binom{n-1}{k-1}}$ and $q = \frac{\beta \log n}{\binom{n-1}{k-1}}$.

We claim that 
$$
\binom{n/2}{d/2}^2 p_{fail}^{(d)} = \begin{cases}
C_1 \cdot n^{-d(I(\alpha,\beta)-1)/2} & \text{if $d < \frac{n}{\log\log n}$} \\
n^{-C_2 \cdot \frac{n}{\log \log n}} & \text{if $d \geq \frac{n}{\log\log n}$}.
\end{cases}
$$ 
for some positive constants $C_1$ and $C_2$ which does not depend on $n$ or $d$. Assuming that the claim is true, we get
$$
\sum_{1\leq d < \frac{n}{\log \log n}} \binom{n/2}{d/2}^2 p_{fail}^{(d)} \leq C_1 \sum_{d \geq 1} (n^{-(I(\alpha,\beta)-1)/2})^d \lesssim n^{-\epsilon/2},
$$
and
$$
\sum_{\frac{n}{\log \log n}\leq d \leq \frac{n}{2}} \binom{n/2}{d/2}^2 p_{fail}^{(d)}
\leq \frac{n}{2} \cdot n^{-C_2 \cdot \frac{n}{\log \log n}} \leq n^{-\omega_n(1)},
$$
hence $p_{fail}^{(d)} = O(n^{-\epsilon/2})$ as desired. 

To complete the proof, we are going to use the tail bound derived in Section \ref{sec:tail}.

Let us first focus on the case that $d \geq \frac{n}{\log\log n}$. We have
\begin{eqnarray*}
N_d &\geq& \frac{2}{2^k \cdot k!} \left((n-2k+2)^k - d^k - (n-d)^k\right) \\
&\geq& \frac{n^k}{2^{k-1} \cdot k!}
\left(\left(1-\frac{2k-2}{n}\right)^k - \frac{1}{(\log \log n)^k} - \left(1 - \frac{1}{\log \log n}\right)^k\right) \\
&=& (1+o_n(1)) \frac{1}{2^{k-1}} \cdot \frac{n}{\log \log n} \cdot \binom{n-1}{k-1}.
\end{eqnarray*}
We get 
$$
p_{fail}^{(d)} \leq \exp\left(-\Omega\left(\frac{n \log n}{\log \log n}\right)\right)
$$
which follows from Theorem \ref{thm:tail}. Since $\binom{n/2}{d/2}^2 \leq 2^n$, we get
$$
\binom{n/2}{d/2}^2 p_{fail}^{(d)} \leq \exp\left(-\Omega\left(\frac{n \log n}{\log \log n}\right) + O(n)\right),
$$
which is still $e^{-\Omega(\frac{n\log n}{\log\log n})}$. 

If $d < \frac{n}{\log \log n}$, then $N_d = (1+o_n(1)) \frac{d}{2^{k-1}} \cdot \binom{n-1}{k-1}$. Using Theorem \ref{thm:tail} with $h(n) = d$, $c_1=1$, $c_2=-1$, $\alpha_1=\alpha$, $\alpha_2=\beta$, $\rho_1=\frac{1}{2^{k-1}}$ and $\rho_2=\frac{1}{2^{k-1}}$, we get
$$
p_{fail}^{(d)} \leq \exp\left(-(1-o_n(1)) I(\alpha,\beta) \cdot d\log n\right).
$$
Since $\binom{n/2}{d/2}^2 \leq n^d$ and $I(\alpha,\beta) > 1$, we have
$$
\binom{n/2}{d/2}^2 p_{fail}^{(d)} 
\leq
\exp\left(-(1-o_n(1))(I(\alpha,\beta)-1) \cdot d\log n\right) \leq C_1 n^{-d\epsilon/2}
$$
for some constant $C_1 >0$ which does not depend on $n$ and it concludes the proof.

\end{proof}

\section{Truncate-and-relax algorithm}
\label{sec:truncrelax}

In this section, we propose an algorithm based on the standard semidefinite relaxation of maximization problem of quadratic function on the hypercube $\{\pm 1\}^V$. We also prove that this algorithm achieves the optimal threshold up to a constant multiplicative factor. We will only focus on the assortative case (i.e. $p > q$) but the algorithm could be adapted for disassortative cases with a different threshold which only depends on $\alpha$ and $\beta$, which we will not derive in this paper. 

Let $H=(V(H),E(H))$ be a $k$-uniform hypergraph, and recall that we defined the weighted graph $(G_H,w)$ on the vertex set $V(H)$ where weights are given by
$$
w(ij) = \#(e \in E(H): \{i,j\} \subseteq e).
$$
We may think $G_H$ be a multigraph realization of $H$, by replacing each hyperedge $e$ in $H$ by the $k$-clique on $e$. For brevity, let us define the adjacency matrix $W$ of $(G_H, w)$ as the symmetric $n$ by $n$ matrix such that its diagonal entries are zero, and $W_{ij} = w(ij)$ for each pair $\{i,j\}\subseteq V$. We defined the estimator $\widehat{\sigma}_{trunc}$ as 
$$
\widehat{\sigma}_{trunc} := \argmax_{x \in \{\pm 1\}^V: \one^T x = 0} \sum_{1\leq i<j\leq n} W_{ij} x_i x_j.
$$

On the other hand, recall that the ML estimator $\widehat{\sigma}_{MLE}$ is the maximizer of 
$$
\inprod{A_H, x^{\oeq k}} = \sum_{e \in \binom{V}{k}} (A_H)_e (x^{\oeq k})_e
$$
over balanced $x$'s. Note that 
\begin{eqnarray*}
(x^{\oeq k})_e &=& \one\{ x[e] = (1,1,\cdots,1)\} + \one\{ x[e] = (-1,-1,\cdots,-1)\} \\
&=& \prod_{v \in e} \left(\frac{1+x_v}{2}\right) + \prod_{v \in e} \left(\frac{1-x_v}{2}\right) \\
&=& \frac{1}{2^{k}} \sum_{I \subseteq e} \left(1+(-1)^{|I|}\right) \prod_{v \in I} x_v \\
&=& \frac{1}{2^{k-1}} \sum_{I \subseteq e: |I| \text{ even}} \prod_{v \in I} x_v.
\end{eqnarray*}
If we collate the terms by its degrees, then we have
$$
\inprod{A_H, x^{\oeq k}} = \inprod{A_H,\one} + 
\frac{1}{2^{k-1}} \sum_{e} \left((A_H)_e\cdot \frac{1}{2^{k-1}} \sum_{\{i,j\} \subseteq e} x_i x_j \right)+ (\text{higher order terms}).
$$
Let $g_H(x)$ be the homogeneous multilinear polynomial of degree 2 defined as
$$
g_H(x) = \sum_{e} (A_H)_e \sum_{\{i,j\} \subseteq e} x_i x_j,
$$
which is a constant multiple of the degree 2 part of $\inprod{A_H,x^{\oeq k}}$. Then,
$$
g_H(x) = \sum_{e} \sum_{\{i,j\} \subseteq e} (A_H)_e x_i x_j = \sum_{1\leq i<j\leq n} W_{ij} x_i x_j.
$$
This justifies the name ``truncate-and-relax''. Instead of computing the maximum of a high-degree polynomial $\inprod{A_H,x^{\oeq k}}$, we first approximate it by a quadratic polynomial $g_H(x)$. Although optimizing $g_H(x)$ over a hypercube is still an NP-hard problem, we consider a convex relaxation of it. It turns out that the optimum of the relaxed problem allows us to recover the ground truth $\pm \sigma$ with high probability if $\alpha$ and $\beta$ satisfies $I_{sdp}(\alpha,\beta)>1$ for some function $I_{sdp}$.

\subsection{Laplacian of the adjacency matrix}

Before we delve into the semidefinite relaxation that our algorithm uses, let us take a detour with a spectral algorithm which can also be thought as a relaxation of $\max g_H(x)$.

Recall that $W$ is the adjacency matrix of a weighted graph. For $x \in \{\pm 1\}^n$ with the corresponding bisection $(A,B)$ where $A = \{v: x_v = 1\}$ and $B = [n]\setminus A$, we have
$$
g_H(x) = \sum_{1\leq i<j\leq n} W_{ij} x_i x_j = \sum_{1\leq i<j\leq n} W_{ij} - 2\sum_{i \in A, j \in B} W_{ij}
$$
so maximizing $g_H(x)$ is equivalent to the minimum bisection problem (\textsf{MIN-BISECTION}):
$$
\textsf{MIN-BISECTION}(W): \min_{x \in \{\pm 1\}^n:\one^T x = 0} \sum_{i \in A, j\in B} W_{ij}.
$$ 

The \emph{Laplacian} of $W$ is a matrix $L_W$ defined as $L_W = D-W$ where $D$ is the diagonal matrix with entries
$$
D_{vv} = \sum_{u \in [n]} W_{uv}.
$$
Equivalently, 
$$
L_W = \sum_{1\leq i<j\leq n} W_{ij} (e_i-e_j)(e_i-e_j)^T
$$
where $e_i$ is the vector with the entry $(e_i)_i = 1$ and zero elsewhere. It implies that
$$
x^T L_W x = \sum_{i,j=1}^n W_{ij} (x_i-x_j)^2 = 4\sum_{i \in A, j \in B} W_{ij},
$$
hence \textsf{MIN-BISECTION}$(W)$ is equivalent to minimizing $x^T L_W x$ over balanced $x$ in $\{\pm 1\}^n$.

By relaxing the condition $x \in \{\pm 1\}^n$ to $\|x\|_2^2 = n$, we get
$$
\min_{x: \|x\|_2^2 = n, \one^T x = 0} x^T L_W x.
$$
Note that $L_W$ is positive semidefinite and the minimum eigenvalue of $L_W$ is zero, since it is diagonally dominant and $\one^T L_W \one = 0$. Hence, the optimal solution of the relaxed problem corresponds to an eigenvector of the second smallest eigenvalue of $L_W$, scaled to have norm $\sqrt{n}$.

It motivates us to suggest a spectral algorithm with the following two stages:
\begin{itemize}
\item (Relaxation) We compute a unit eigenvector $\xi$ of the second smallest eigenvalue of $L_W$. 
\item (Rounding) We round $\sqrt{n} \cdot \xi$ to the closest point on $\{\pm 1\}^n$, which corresponds to $x$ with $x_v = \mathrm{sgn}(\xi_v)$. 
\end{itemize}
We remark that in \cite{ghoshdastidar2014consistency, ghoshdastidar2015spectral, ghoshdastidar2017consistency}, the authors generalize this idea to the case when we have three or more clusters. Their algorithm computes eigenvectors of $k$ smallest eigenvalues to associate each vertex with a vector in $\RR^k$, and uses $k$-means clustering on them to find the community label of each vertex. Their algorithm has a few advantages such as that it applies to weak recovery and detection as well as exact recovery, that it generalizes to non-uniform models, and that it runs in nearly-linear time in $n$. However, it only works in a order-wise suboptimal parameter regime, requiring $p$ and $q$ be at least $\Omega\left(\frac{\log^2 n}{n^{k-1}}\right)$ for exact recovery.

Subsequently, in \cite{ahn2018hypergraph} and independently in \cite{chien2018minimax}, spectral algorithms with an additional local refinement step were proposed, and it was proved that both algorithms achieve exact recovery in the regime where $p$ and $q$ are $\Omega\left(\frac{\log n}{n^{k-1}}\right)$, which matches the statistical limit in terms of order in $n$. Also, we note that it was mentioned in \cite{abbe2016community} that the local refinement technique used for the SBM can be extended to the hypergraph case, together with a partial recovery algorithm in \cite{angelini2015spectral}. Finally, we remark that it was proved in \cite{chien2018minimax} that their algorithm achieves the statistical limit shown in this paper. In other words, there is an efficient algorithm which recovers the ground truth almost asymptotically surely whenever $I(\alpha,\beta)>1$.

\subsection{Semidefinite relaxation and its dual}

Let $X = xx^T$. Then, the condition that $x \in \{\pm 1\}^V$ and $\one^T x = 0$ is equivalent to that $X$ is a symmetric $n$ by $n$ positive semidefinite rank-one matrix such that $X_{ii} = 1$ for all $i \in V$ and $\one^T X \one = 0$. If we relax the rank condition, then we get the following semidefinite program equivalent to (\ref{eqn:sdp}) as argued in the introduction:
\begin{equation}
\label{eqn:TR}
\begin{aligned}
\text{maximize} & & & \inprod{W, X} \\
\text{subject to} & & & X_{ii} = 1 \text{ for all $i \in [n]$}\\
& & & \inprod{X, \one\one^T} = 0 \\
& & & X\succeq 0.
\end{aligned}
\end{equation}
The dual of (\ref{eqn:TR}) is
\begin{equation}
\label{eqn:TRdual}
\begin{aligned}
\text{minimize} & & & tr(D) \\
\text{subject to} & & & D \text{ is $n\times n$ diagonal matrix, $\lambda \in \RR$}\\
& & & D+\lambda \one\one^T - W \succeq 0.
\end{aligned}
\end{equation}
We recall that $\widehat{\Sigma}$ was defined as the optimum solution of the primal program (\ref{eqn:TR}), and we say $\widehat{\Sigma}$ recovers the ground truth $\sigma$ if $\widehat{\Sigma} = \sigma\sigma^T$. In the case of $k=2$ (the usual SBM), it is known that the relaxed SDP achieves exact recovery up to the statistical threshold even without the local refinement step \cite{hajek2016achieving}. We prove that for any $k \geq 2$ our algorithm successfully recover the ground truth, as long as $I_{sdp}(\alpha,\beta) > 1$ which is slightly weaker than the statistical threshold $I(\alpha,\beta)>1$. On the other hand, we show that for $k \geq 4$ our algorithm fails with probability $1-o_n(1)$ for some $(\alpha,\beta)$ even when exact recovery is statistically posible (see the next section).

Let $X$ be an optimum solution of the primal and let $(D, \lambda)$ be an optimum solution for dual. Then by complementary slackness we get $\inprod{X, D+\lambda\one\one^T - W} = 0$. Conversely, if $X$ is a feasible solution for the primal and $(D, \lambda)$ is a feasible solution for the dual, then $X$ and $(D,\lambda)$ are optimal if $\inprod{X, D+\lambda\one\one^T - W} = 0$. It implies that $X=\sigma\sigma^T$ is optimal if there exists dual feasible solution $(D,\lambda)$ such that
$$
\inprod{\sigma\sigma^T, D+\lambda\one\one^T - W} = 0,
$$
which is equivalent to 
$$
D_{ii} = \sum_{j \in V} W_{ij} \sigma_i \sigma_j
$$
since $D+\lambda\one\one^T - W$ is positive semidefinite. Note that $D$ is completely determined by $W$ and $\sigma$. 

Let $\Gamma = \diag(\sigma) W \diag(\sigma)$ and $D_\Gamma = \diag(\Gamma \one)$. Note that $D_\Gamma$ is equal to $D$ defined above. Let $L_\Gamma = D_\Gamma - \Gamma$. Then, 
$$
L_\Gamma = \diag(\sigma) (D_\Gamma - W) \diag(\sigma)
$$
by definition. 

\begin{proposition}
Let $\Pi$ be the projector matrix onto the orthogonal complement of the span of $\{\one, \sigma\}$, i.e.,
$$
\Pi = I - \frac{1}{n} \one\one^T - \frac{1}{n} \sigma\sigma^T.
$$
If $\Pi L_\Gamma \Pi$ is positive semidefinite, then $\sigma\sigma^T$ is an optimal solution for (\ref{eqn:TR}). Moreover, if the third smallest eigenvalue of $\Pi L_\Gamma \Pi$ is positive, then $\sigma\sigma^T$ is the unique optimum. 
\end{proposition}

\begin{proof}
First note that $(D_\Gamma, \lambda)$ is a feasible solution for the dual if there exists $\lambda$ such that $D_\Gamma - W+\lambda \one\one^T$ is positive semidefinite. By multiplying $\diag(\sigma)$ on the both side, it is equivalent to that $L_\Gamma + \lambda \sigma\sigma^T$ is positive semidefinite for some $\lambda$. This condition is satisfied if and only if $\Pi L_\Gamma \Pi \succeq 0$ and hence $\sigma\sigma^T$ is an optimal solution for the primal. 

Moreover, if $\lambda_3(\Pi L_\Gamma \Pi) > 0$ then there exists $\lambda$ such that $L_\Gamma + \lambda \sigma\sigma^T$ is positive definite on the orthogonal complement of $\one$. It immediately implies that $\sigma\sigma^T$ is the unique optimal solution for the primal.
\end{proof}

In the remainder of this section, we present and prove a sufficient condition for $\lambda_3(\Pi L_\Gamma \Pi)>0$. We also present and prove a necessary condition for $\widehat{\sigma}_{trunc}$ being $\sigma$ up to global sign flip with high probability. 

\subsection{Performance of the algorithm}

We first present the main result of this section.
\begin{theorem}
\label{thm:TRupper}
Suppose that $\alpha$ and $\beta$ satisfies
$$
\frac{k-1}{2^{k-1}} (\alpha-\beta) > \sqrt{ 4(k-1)\left(\frac{k}{2^{k}} \alpha+\left(1-\frac{k}{2^k}\right)\beta\right)}.
$$
Then, $\lambda_3(\Pi L_\Gamma \Pi)>0$ with probability $1-o_n(1)$.  
\end{theorem}

We remark that Theorem \ref{thm:TRupper} implies Theorem \ref{thm:B}. To prove \ref{thm:TRupper}, we are going to use standard concentration result for the norm of the sum of random matrices. We first note that
\begin{eqnarray*}
\Pi L_\Gamma \Pi &=& \E \Pi L_\Gamma \Pi + (\Pi L_\Gamma \Pi - \E \Pi L_\Gamma \Pi) \\
&=& \Pi (\E L_\Gamma) \Pi + \Pi(L_\Gamma - \E L_\Gamma)\Pi.
\end{eqnarray*}
We would like to prove that if $\alpha$ and $\beta$ satisfies the condition in \ref{thm:TRupper}, then with probability $1-o_n(1)$,
$$
\lambda_3\left(\Pi (\E L_\Gamma) \Pi\right) > \| \Pi(L_\Gamma - \E L_\Gamma)\Pi\|.
$$

Let $\one_e$ be the vector in $\RR^V$ where $(\one_e)_i = 1$ when $i \in e$ and $(\one_e)_i = 0$ otherwise. Let $\sigma_e = \diag(\sigma)\one_e$. We note that 
$$
Q_\cH = \sum_{e} (A_\cH)_e \left(\one_e\one_e^T - \diag(\one_e)\right),
$$
and so
$$
\Gamma = \sum_{e} (A_\cH)_e \left(\sigma_e\sigma_e^T - \diag(\one_e)\right).
$$
It implies that 
\begin{eqnarray*}
L_\Gamma &=& \sum_{e} (A_\cH)_e L_{\sigma_e\sigma_e^T - \diag(\one_e)} \\
&=& \sum_{e} (A_\cH)_e \left((\one_e^T \sigma_e)\diag(\sigma_e)-\sigma_e\sigma_e^T\right).
\end{eqnarray*}

\begin{proposition}
$$
\Pi (\E L_\Gamma) \Pi = \frac{p-q}{2} n \binom{n/2-2}{k-2} \cdot \Pi,
$$
hence $\lambda_3(\Pi (\E L_\Gamma) \Pi) = \frac{k-1}{2^{k-1}} (\alpha-\beta) \log n + o(\log n)$. 
\end{proposition}

\begin{proof}
Note that $\E L_\Gamma$ is invariant under the permutation of $V$ which preserves $\sigma$. Hence, we can write $\E L_\Gamma$ as 
$$
\E L_\Gamma = \frac{\inprod{\E L_\Gamma, \Pi}}{n-2}\Pi + \frac{\inprod{\E L_\Gamma, \one\one^T}}{n^2} \one\one^T + \frac{\inprod{\E L_\Gamma, \sigma\sigma^T}}{n^2} \sigma\sigma^T.
$$
We have $\inprod{\E L_\Gamma, \one\one^T} = \one^T (\E L_\Gamma) \one = 0$ by definition of $L_\Gamma$. Also,
\begin{eqnarray*}
\inprod{\E L_\Gamma, \sigma\sigma^T} &=& 
\sum_{e} \E(A_\cH)_e \cdot \sigma^T((\one_e^T \sigma_e)\diag(\sigma_e)-\sigma_e\sigma_e^T)\sigma \\
&=& \sum_{e} \E(A_\cH)_e ((\one_e^T \sigma_e)^2 - k^2) \\
&=& q \sum_{r=1}^{k-1} ((k-2r)^2-k^2)\binom{n/2}{r}\binom{n/2}{k-r} \\
&=& -4q \sum_{r=1}^{k-1} r(k-r)\binom{n/2}{r}\binom{n/2}{k-r} \\
&=& -4q \cdot \left(\frac{n}{2}\right)^2 \binom{n-2}{k-2} \\
&=& -qn^2 \binom{n-2}{k-2}.
\end{eqnarray*}
On the other hand,
\begin{eqnarray*}
tr(\E L_\Gamma) &=&
\sum_{e} \E(A_\cH)_e \cdot tr((\one_e^T \sigma_e)\diag(\sigma_e)-\sigma_e\sigma_e^T) \\
&=& \sum_{e} \E(A_\cH)_e \cdot ((\one_e^T \sigma_e)^2 - k) \\
&=& (k^2-k)\left(q\binom{n}{k}+2(p-q)\binom{n/2}{k}\right) - qn^2 \binom{n-2}{k-2} \\
&=& 2(p-q)(k^2-k) \binom{n/2}{k} - qn\binom{n-2}{k-2}.
\end{eqnarray*}
Hence,
\begin{eqnarray*}
\inprod{\E L_\Gamma, \Pi} &=& tr(\E L_\Gamma) - \frac{1}{n} \inprod{\E L_\Gamma, \sigma\sigma^T} \\
&=& 2(p-q)(k^2-k) \binom{n/2}{k} - qn\binom{n-2}{k-2} + qn\binom{n-2}{k-2} \\
&=& \frac{p-q}{2} \cdot n(n-2)\binom{n/2-2}{k-2}.
\end{eqnarray*}
We get
$$
\Pi \E L_\Gamma \Pi = \frac{\inprod{\E L_\Gamma, \Pi}}{n-2}\Pi = \frac{p-q}{2} n\binom{n/2-2}{k-2} \cdot \Pi.
$$
\end{proof}

Now, let us bound the operator norm of $\Pi(L_\Gamma - \E L_\Gamma)\Pi$. We need the following version of Matrix Bernstein inequality \cite{tropp2012user}.

\begin{theorem}[Matrix Bernstein inequality]
Let $\{X_k\}$ be a finite sequence of independent, symmetric random matrices of dimension $N$. Suppose that $\E X_k = 0$ and $\|X_k\| \leq M$ almost surely for all $k$. Then for all $t \geq 0$,
$$
\Pr\left(\left\| \sum_k X_k \right\| \geq t\right) \leq N\cdot \exp\left(-\frac{t^2/2}{\sigma^2 + Mt/3}\right)
\quad \text{where } \sigma^2 = \left\|\sum_k \E X_k^2 \right\|.
$$
\end{theorem}

Recall that 
$$
L_\Gamma = \sum_{e} (A_\cH)_e \left((\one_e^T \sigma_e)\diag(\sigma_e)-\sigma_e\sigma_e^T\right).
$$
Hence, 
$$
\Pi(L_\Gamma - \E L_\Gamma)\Pi = 
\sum_{e} \left((A_\cH)_e-\E(A_\cH)_e\right)\cdot \Pi\left((\one_e^T \sigma_e)\diag(\sigma_e)-\sigma_e\sigma_e^T\right)\Pi.
$$
We note that 
$$
\|\Pi\left((\one_e^T \sigma_e)\diag(\sigma_e)-\sigma_e\sigma_e^T\right)\Pi\| \leq |\one_e^T \sigma_e| + \|\sigma_e\|^2 \leq 2k
$$
for any $e \in \binom{V}{k}$. By Matrix Bernstein inequality, we have
$$
\Pr\left(\| \Pi(L_\Gamma - \E L_\Gamma)\Pi \| \geq t\right)
\leq n \cdot \exp\left(-\frac{t^2/2}{\sigma^2 + 2kt/3}\right)
$$
where 
$$
\sigma^2 = \left\|
\sum_{e} \E\left((A_\cH)_e-\E(A_\cH)_e\right)^2 \cdot \left(\Pi\left((\one_e^T \sigma_e)\diag(\sigma_e)-\sigma_e\sigma_e^T\right)\Pi\right)^2
\right\|.
$$
If $\sigma = \omega_n(1)$, then letting $t = (1+\epsilon)\sigma \sqrt{2\log n}$ we have
\begin{eqnarray*}
\Pr\left(\| \Pi(L_\Gamma - \E L_\Gamma)\Pi \| \geq (1+\epsilon)\sigma \sqrt{2\log n}\right)
&\leq& n \cdot \exp\left(-(1+\epsilon)^2\log n + o(\log n)\right) \\
&\leq& n^{-2\epsilon + o(1)}.
\end{eqnarray*}

\begin{proposition}
\label{prop:sigmasq}
$$
\sigma^2 \leq  2(k-1)\left(\frac{k}{2^{k}} \alpha+\left(1-\frac{k}{2^k}\right)\beta\right)\log n + O\left(\frac{\log n}{n}\right).
$$
\end{proposition}

\begin{proof}
Let $Y_e$ be
$$
Y_e := \left((\one_e^T \sigma_e)\diag(\sigma_e)-\sigma_e\sigma_e^T\right) \Pi \left((\one_e^T \sigma_e)\diag(\sigma_e)-\sigma_e\sigma_e^T\right)
$$
and let $\Sigma = \sum_{e} \E(A_\cH)_e Y_e$. Since $Y_e$ is positive semidefinite and $\E\left((A_\cH)_e-\E(A_\cH)_e\right)^2 \leq \E(A_\cH)_e$, we have
\begin{eqnarray*}
\sigma^2 &=& \left\| \sum_{e} \E\left((A_\cH)_e-\E(A_\cH)_e\right)^2 \cdot \Pi Y_e \Pi \right\| \\
&\leq& \left\| \sum_{e} \E(A_\cH)_e \cdot \Pi Y_e \Pi \right\| = \left\| \Pi \Sigma \Pi \right\|. \\
\end{eqnarray*}

We get the exact expression of $\Sigma$ in the following lemma. We defer the proof to the section \ref{sec:variance} in the appendix.

\begin{lemma}
\label{lemma:variance}
$\Sigma = c_1 \cdot \frac{1}{n} \sigma\sigma^T + c_2 \cdot \Pi$ where
$$
c_1 = (k-1)\beta \log n + O\left(\frac{\log n}{n}\right)
$$
and
$$
c_2 = 2(k-1)\left(\frac{k}{2^{k}} \alpha+\left(1-\frac{k}{2^k}\right)\beta\right)\log n + O\left(\frac{\log n}{n}\right)
$$
\end{lemma}

The lemma implies that $\Pi\Sigma\Pi = c_2 \Pi$, so the norm of $\Pi\Sigma\Pi$ is equal to $c_2$. Since $\sigma^2 \leq \|\Pi\Sigma\Pi\|$ as we argued above, we get $\sigma^2 \leq c_2$ as desired. 
\end{proof}

We are now ready to prove Theorem \ref{thm:TRupper}.

\begin{proof}[Proof of Theorem \ref{thm:TRupper}]

Let $\epsilon$ be an arbitrary positive real number. By Matrix Bernstein Inequality and Proposition \ref{prop:sigmasq}, with probability $1-O(n^{-2\epsilon+o_n(1)})$ we have 
$$
\| \Pi(L_\Gamma - \E L_\Gamma)\Pi \| < (1+\epsilon) \sqrt{2c_2\log n}
$$
where
$$
c_2 = (k-1)\left(k(1-2^{-k+1})\alpha + (2^{-k+1}-k+2)\beta\right)\log n + o(\log n).
$$
It implies that $\lambda_3(\Pi L_\Gamma \Pi) > 0$ with probability $1-O(n^{-2\epsilon+o_n(1)})$ if
$$
\frac{k-1}{2^{k-1}} (\alpha-\beta) \log n + o(\log n) > (1+\epsilon) \sqrt{2c_2\log n}.
$$
It reduces to 
$$
\frac{k-1}{2^{k-1}} (\alpha-\beta) > (1+\epsilon) \sqrt{ 4(k-1)\left(\frac{k}{2^{k}} \alpha+\left(1-\frac{k}{2^k}\right)\beta\right)}.
$$
\end{proof}

\subsection{Limitation of the algorithm}

In the previous section, we proved that the truncate-and-relax algorithm successfully recovers $\sigma$ with high probability if $\alpha$ and $\beta$ satisfies
$$
\frac{k-1}{2^{k-1}} (\alpha-\beta) > \sqrt{ 4(k-1)\left(\frac{k}{2^{k}} \alpha+\left(1-\frac{k}{2^k}\right)\beta\right)}.
$$
It is natural to ask whether this bound is improvable or not. Recall that $\widehat{\sigma}_{trunc}$ is the optimum solution for $\max x^T W x$ over balanced $x$'s in $\{\pm 1\}^n$. Since our algorithm is the relaxed version of it, we have 
$$
\Pr\left(\widehat{\Sigma} \neq \sigma\sigma^T\right) \geq \Pr\left(\widehat{\sigma}_{trunc}(\cH) \not\in \{\sigma,-\sigma\}\right).
$$
The following theorem gives a condition on $\alpha$ and $\beta$ such that that the probability that $\widehat{\sigma}_{trunc}(\cH)$ fails to recover $\sigma$ is $1-o_n(1)$. 

\begin{theorem}
\label{thm:tailbound2}
Let 
$$
I_2(\alpha,\beta) = \max_{t \geq 0} \frac{1}{2^{k-1}} \left(\alpha(1-e^{-(k-1)t}) + \beta\sum_{r=1}^{k-1} \binom{k-1}{r} (1-e^{-(k-1-2r)t})\right).
$$
If $I_2(\alpha,\beta) < 1$, then
$$
\Pr\left(\widehat{\sigma}_{trunc}(\cH) \not\in \{-\sigma,\sigma\}\right) = 1-o_n(1).
$$
In particular, the truncate-and-relaxation algorithm fails to recover $\sigma$ with probability $1-o_n(1)$.
\end{theorem}

\begin{proof}
The proof is a slight modification of the proof of Theorem \ref{thm:A}. Essentially it reduces to prove that
$$
\Pr\left(X_a \leq -\frac{2\log n}{\log \log n}\right) \geq n^{-I_2(\alpha,\beta)-o_n(1)}.
$$
where
$$
X_{a} = \sum_{e: e \cap U = \{a\}} (A_\cH)_e \left(\sum_{\{i,j\} \subseteq e} \sigma_i\sigma_j - \sigma^{(a)}_i \sigma^{(a)}_j \right)
$$
and $U = U_A \cup U_B$ and $a \in U_A$ are defined as in the proof of Theorem \ref{thm:A}, and this tail bound follows from the Theorem \ref{thm:tail}. Details are deferred to Section \ref{sec:tailbound2} in the appendix.

\end{proof}

\section{Discussion}
\label{sec:conjecture}

Let us first recapitulate the main results of this paper. In the stochastic block model for $k$-uniform hypergraphs where the (hyper)edge probabilities are given as
$$
p = \frac{\alpha \log n}{\binom{n-1}{k-1}} \quad \text{and} \quad 
q = \frac{\beta \log n}{\binom{n-1}{k-1}}
$$
for some constants $\alpha$ and $\beta$ such that $\alpha > \beta > 0$, we observed the following phase transition behaviours on exact recovery problem:
\begin{itemize}
\item[(i)] If $I(\alpha,\beta)<1$, then exact recovery is not possible. Conversely, if $I(\alpha,\beta)<1$ then the ML estimator recovers the correct partition (up to a global sign flip) with probability $1-o_n(1)$. 
\item[(ii)] If $I_{sdp}(\alpha,\beta)>1$, then the truncate-and-relax algorithm recovers the partition (up to a global sign flip) with probability $1-o_n(1)$. 
\item[(iii)] If $I_2(\alpha,\beta)<1$, then the truncate-and-relax algorithm fails with probability $1-o_n(1)$.
\end{itemize}
Here $I$, $I_2$ and $I_{sdp}$ are functions depending on $\alpha$ and $\beta$ (and implicitly depending on $k$, which we assumed to be a constant) defined as
\begin{eqnarray*}
I(\alpha,\beta) &=& \frac{1}{2^{k-1}} \left(\sqrt{\alpha}-\sqrt{\beta}\right)^2 \\
I_2(\alpha,\beta) &=& \max_{t \geq 0} \frac{\alpha-\beta}{2^{k-1}} (1-e^{-(k-1)t}) + \beta\left(1-\left(\frac{e^t+e^{-t}}{2}\right)^{k-1}\right)\\
I_{sdp}(\alpha,\beta) &=& \frac{k-1}{2^{2k}} \frac{(\alpha-\beta)^2}{\left(\frac{k}{2^{k}} \alpha+\left(1-\frac{k}{2^k}\right)\beta\right)}.
\end{eqnarray*}

We first note that sharp phase transition occurs at $I(\alpha,\beta)=1$ for exact recovery. Indeed, it can be \emph{efficiently} achieved, by spectral algorithms with a local refinement step as suggested in \cite{chien2018minimax}. Specifically authors of \cite{chien2018minimax} prove that their algorithm achieves exact recovery whenever $I(\alpha,\beta)>1$ and conjectured that $I(\alpha,\beta)=1$ is the sharp threshold. We confirmed their conjecture in this work. On the other hand, there is a gap between the guaranteed performance of the truncate-and-relax algorithm and the impossibility region of the algorithm as shown in Figure \ref{fig:thresholds} and Figure \ref{fig:simulation}. We are yet to show how the algorithm works in between, which is when $\alpha$ and $\beta$ satisfies $I_{sdp}(\alpha,\beta) < 1$ but $I_2(\alpha,\beta) > 1$. We propose that the line $I_2(\alpha,\beta)=1$ is the correct threshold for the performance guarantee of the algorithm.

\begin{conjecture}
If $I_2(\alpha,\beta) > 1$, then the truncate-and-relax algorithm successfully recovers $\sigma\sigma^T$ with probability $1-o_n(1)$. 
\end{conjecture}

There are a few reasons to believe this conjecture. First, if we look deeper into the proof of Theorem \ref{thm:B} then the main obstacle to prove the conjecture arises from when we use the matrix Bernstein inequality to bound $\|\Pi (L_\Gamma - \E L_\Gamma)\Pi\|$. The matrix Bernstein inequality gives us that
$$
\E \|\Pi (L_\Gamma - \E L_\Gamma)\Pi\| \lesssim \sigma \sqrt{\log n}
$$
where
$$
\sigma^2 = \left\| \E(\Pi (L_\Gamma - \E L_\Gamma)\Pi)^2 \right\|.
$$
In the case of $k=2$, the random matrix $\Gamma$ has independent entries and one can obtain a tighter bound for $\|L_\Gamma - \E L_\Gamma\|$, via combinatorial method \cite{feige2005spectral}, stochastic comparison argument \cite{hajek2016achieving}, or trace method \cite{Bandeira2016}. Also, in \cite{Bandeira2016laplacian} the following bound for Laplacian random matrices was proved.

\begin{theorem}
Let $L$ be a $n \times n$ symmetric random Laplacian matrix (i.e. satisfying $L \one = 0$) with centered independent off-diagonal entries such that $\sum_{j \in [n]\setminus \{i\}} \E L_{ij}^2$ is equal for all $i$, and
$$
\sum_{j \in [n]\setminus \{i\}} \E L_{ij}^2 \gtrsim \max_{i\neq j} \|L_{ij}\|_{\infty}^2 \log n.
$$
Then, with high probability,
$$
\|L\| \lesssim \left(1+\frac{1}{\sqrt{\log n}}\right) \max_{i} L_{ii}.
$$
\end{theorem}

This bound cannot be used for $k > 2$ as entries of $\Gamma$ are not independent to each other. We ask whether the bound could be extended to our setting: would we have similar bound if $L$ can be expressed as
$$
L = \sum_{S \subseteq [n], |S|=k} \xi_S L^{(S)},
$$
where $L^{(S)}$ is $n \times n$ symmetric Laplacian matrix such that $L_{ij}^{(S)}$ is non-zero only if $i, j \in S$?

\begin{figure}
\centering
\includegraphics[scale=0.6]{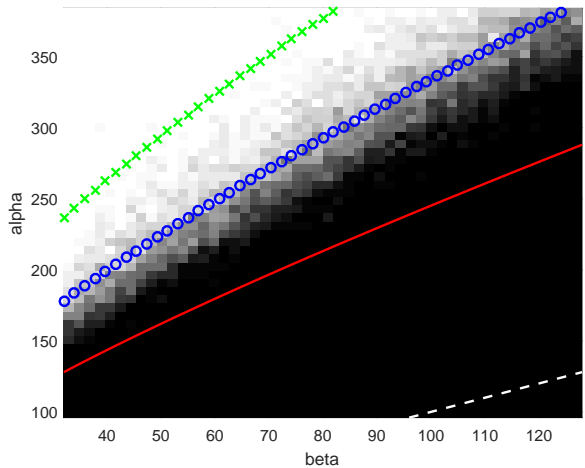}
\caption{Result of simulation of the truncate-and-relax algorithm for $k=6$ and $n=500$. Each gray-scale block corresponds to a pair $(\alpha,\beta)$, and its color denotes the success rate over 30 trials (black corresponds to 0 success, and brighter color correspond to higher success rate). The solid line represents $I(\alpha,\beta)=1$, the circled line represents $I_2(\alpha,\beta)=1$, and the x-marked line represents $I_{sdp}(\alpha,\beta)=1$.}
\label{fig:simulation}
\end{figure}

We ran a simulation to support our conjecture. For each $\alpha$ and $\beta$, we generated 30 random hypergraphs according to the model, and constructed the dual certificate for each hypergraph as in the proof of Theorem \ref{thm:B}. When the constructed dual solution is positive-semidefinite, it was counted as a success. Figure \ref{fig:simulation} shows the result of the simulation and it suggests that the true phase transition occurs at $I_2(\alpha,\beta)=1$ as we proposed.

\bibliographystyle{alpha}
\bibliography{mybib}

\appendix
\section{Tail probability of weighted sum of binomial variables}
\label{sec:tail}
In this section, we investigate the precise asymptotics of the tail probability of weighted sum of independent binomial variables. Using Theorem \ref{thm:tail}, we derive the formulas which were used to prove information-theoretic limits in section \ref{sec:IT} and \ref{sec:truncrelax}.

\begin{theorem}
\label{thm:tail}
Let $r$ and $s$ be positive integers. Let $c_1, c_2, \cdots, c_r$ be non-zero real numbers. Let $h(n)$ be a non-decreasing function which is $\Omega(1)$ and $o(n^{s/2}/\log n)$. For each $i \in [r]$, let $Y_i$ be the random variable distributed as the binomial distribution $\mathrm{Bin}(N_i, p_i)$ where 
$$
N_i = (1+o_n(1)) \rho_i \cdot h(n) \binom{n}{s} \quad \text{and} \quad
p_i = (1+o_n(1)) \alpha_i \cdot \frac{\log n}{\binom{n}{s}},
$$
for some positive constant (not depending on $n$) $\rho_i$ and $\alpha_i$. Let $X = \sum_{i=1}^r c_i Y_i$. Suppose that (i) not all $c_i$ are positive, and (ii) $\sum_{i=1}^r c_i \alpha_i \rho_i > 0$. Then, for any $\delta \in (-\infty, \sum_{i=1}^r c_i \alpha_i \rho_i)$, we have
$$
\Pr(X \leq (1+o_n(1))\delta \cdot h(n) \log n) = \exp\left(-(1+o_n(1)) I^* \cdot h(n) \log n\right),
$$
where
$$
I^* = \max_{t \geq 0} \left(-t\delta+ \sum_{i=1}^r \alpha_i \rho_i (1-e^{-tc_i})\right).
$$
\end{theorem}

\begin{proof}
Let us first prove the upper bound on the tail probability. Let $x = (1+o_n(1))\delta \cdot h(n) \log n$. Then, by Chebyshev-type inequality, for any $t \geq 0$ we have
\begin{eqnarray*}
\Pr(X \leq x) &\leq& \frac{\E e^{-tX}}{e^{-tx}} \\
&=& e^{tx} \prod_{i=1}^r \E e^{-tc_i Y_i} \\
&=& e^{tx} \prod_{i=1}^r \left(1 - p_i(1-e^{-c_i t})\right)^{N_i} \\
&\leq& \exp\left(tx - \sum_{i=1}^r N_i p_i (1-e^{-c_i t})\right) \\
&=& \exp\left(-(1+o_n(1)) h(n)\log n \cdot \left(-t\delta + \sum_{i=1}^r \alpha_i \rho_i (1-e^{-c_i t})\right)\right).
\end{eqnarray*}
Here the fourth inequality follows from $1-x \leq e^{-x}$. By optimizing over $t \geq 0$, we get the desired bound. 

To prove the lower bound, note that 
$$
\Pr(X \leq x) \geq \prod_{i=1}^r \Pr(Y_i = y_i) = \prod_{i=1}^r \binom{N_i}{y_i} p_i^{y_i} (1-p_i)^{N_i-y_i}
$$
for any positive integers $y_1,\cdots,y_r$ satisfying $\sum_{i=1}^r c_i y_i \leq x$. 

Let $\phi(t) = -\delta t + \sum_{i=1}^r \alpha_i \rho_i (1-e^{-c_i t})$ and let $t^*$ be the maximizer of $\phi(t)$. Note that $\phi(t)$ is strictly convex, as 
$$
\phi''(t) = \sum_{i=1}^r c_i^2 \alpha_i \rho_i e^{-c_i t} > 0
$$
for any $t \geq 0$. Moreover, $\phi'(0) = \sum_{i=1}^r c_i \alpha_i \rho_i - \delta > 0$ and $\lim_{t \to \infty} \phi'(t) = -\infty$. Hence, there exists unique $t^*$ satisfying $\phi'(t^*) = 0$, which is the maximizer of $\phi(t)$. 

Let $\tau_i = \alpha_i \rho_i e^{-c_i t^*}$ for $i \in [r]$ and let $y_1,\cdots,y_r$ be integers such that $\sum_{i=1}^r c_i y_i \leq x$ and $y_i = (1-o_n(1)) \tau_i \cdot h(n) \log n$. Such $y_i$'s exist because 
$$
\sum_{i=1}^r c_i \tau_i = \sum_{i=1}^r c_i \alpha_i \rho_i e^{-c_i t^*} = \delta.
$$

We are going to use the following bound for the binomial coefficient for $\ell \leq \sqrt{N}$
$$
\binom{N}{\ell} \geq \frac{N^\ell}{4 \cdot \ell!}.
$$
By Stirling's approximation, we have $\ell! \leq e\sqrt{\ell} \cdot \left(\frac{\ell}{e}\right)^\ell$ so 
$$
\log \binom{N}{\ell} \geq \ell \log \left(\frac{eN}{\ell}\right)-\log(4e\sqrt{\ell}).
$$
Note that $y_i \ll \sqrt{N_i}$ since $h(n) = o(n^{s/2}/\log n)$. Hence,
$$
\log \left(\binom{N_i}{y_i} p_i^{y_i} (1-p_i)^{N_i-y_i}\right)
\geq
y_i \log\left(\frac{eN_ip_i}{(1-p_i)y_i}\right) + N_i \log (1-p_i) - \log(4e\sqrt{y_i}).
$$
Moreover,
\begin{eqnarray*}
y_i \log\left(\frac{eN_ip_i}{(1-p_i)y_i}\right) &=& (1-o_n(1))h(n) \log n \cdot \tau_i \log \left(\frac{e\alpha_i \rho_i}{\tau_i}\right)\\
N_i \log (1-p_i) &=& -(1+o_n(1))\alpha_i \rho_i \cdot h(n)\log n \\
\log(4e\sqrt{y_i}) &=& o(h(n)\log n).
\end{eqnarray*}
We get
$$
\Pr(X\leq x) \geq \exp\left(-(1+o_n(1)) h(n)\log n \cdot \sum_{i=1}^r \left(\alpha_i \rho_i - \tau_i \log \left(\frac{e\alpha_i \rho_i}{\tau_i}\right) \right)\right).
$$
Plugging in $\tau_i = \alpha_i \rho_i e^{-c_i t^*}$, we get
\begin{eqnarray*}
\sum_{i=1}^r \left(\alpha_i \rho_i - \tau_i \log \left(\frac{e\alpha_i \rho_i}{\tau_i}\right) \right)
&=& \sum_{i=1}^r \left(\alpha_i \rho_i - \alpha_i \rho_i (1+c_i t^*)e^{-c_i t^*}\right) \\
&=& -t^* \sum_{i=1}^r c_i\alpha_i \rho_i e^{-c_i t^*} + \sum_{i=1}^r \alpha_i \rho_i (1-e^{-c_i t^*})\\
&=& -\delta t^* + \sum_{i=1}^r \alpha_i \rho_i (1-e^{-c_i t^*})\\
&=& I^*,
\end{eqnarray*}
where the third equality follows from that $\sum_{i=1}^r c_i \alpha_i \rho_i e^{-c_i t^*} = \delta$. Hence,
$$
\Pr(X\leq x) \geq \exp\left(-(1+o_n(1))I^* \cdot h(n) \log n\right)
$$
as desired. 
\end{proof}

We remark that the condition $h(n) = o(n^{s/2}/\log n)$ is not required for the upper bound. 

\subsection{Proof of Lemma \ref{lemma:tailbound0}}
\label{sec:tailbound0}

Let us restate the lemma for readers.
\begin{lemma}[Lemma \ref{lemma:tailbound0}]
Let $X$ be a sum of independent Bernoulli variables such that $\E X = \Theta(\gamma \log n)$ where $\gamma = o_n(\log^{-1} n)$. Let $\delta$ be a positive number which decays to 0 as $n$ grows, with $\delta = \omega_n(\log^{-1} n)$. Then,
$$
\Pr\left(X > \delta \log n \right) \leq n^{-\delta \log \frac{\delta}{\gamma} + o_n(1)}.
$$
\end{lemma}

A standard Chernoff's bound implies that
$$
\Pr(X \geq t\E X) \leq \left(\frac{e^{t-1}}{t^t}\right)^{\E X}
$$
for any $t \geq 0$. Let $t = \frac{\delta \log n}{\E X}$. Then,
\begin{eqnarray*}
\Pr(X \geq \delta \log n) &\leq& \exp\left((t-1-t\log t)\E X\right) \\
&=& \exp\left( \left(1-\log \frac{\delta\log n}{\E X}\right) \delta \log n - \E X\right) \\
&=& \exp\left( \left(1-\log \frac{\delta}{\gamma} + O(1)\right) \delta \log n - o_n(1)\right) \\
&=& n^{-\delta \log \frac{\delta}{\gamma} + o_n(1)}
\end{eqnarray*}
as desired.

\subsection{Proof of Lemma \ref{lemma:tailbound1}}
\label{sec:tailbound1}

Let $a \in U_A$. Recall that
$$
X_{a} = \sum_{e: e \cap U = \{a\}} c_e (A_\cH)_e,
$$
where $c_e = \Sigma_e - \Sigma^{(ab)}_e$ for any $b \in U_B$. Concretely, the value of $c_e$ for $e$ satisfying $e \cap U = \{a\}$ is determined by the size of intersection of $e\setminus \{a\}$ and $A \setminus U_A$ as follows:
$$
c_e = \begin{cases}
+1 & \text{if $e\setminus \{a\} \subseteq A\setminus U_A$} \\
-1 & \text{if $(e\setminus\{a\}) \cap (A\setminus U_A) = \emptyset$}\\
0 & \text{otherwise}.
\end{cases}
$$
Hence, $X_a = Y_1 - Y_2$ where $Y_1$ and $Y_2$ are independent random variables such that $Y_1 \sim \mathrm{Bin}(N_1, p)$ and $Y_2 \sim \mathrm{Bin}(N_2, q)$ with
$$
N_1 = \binom{n/2-|U_A|}{k-1} \quad \text{and} \quad N_2 = \binom{n/2-|U_B|}{k-1}.
$$

Using Theorem \ref{thm:tail} with $c_1 = 1$, $c_2=-1$, $\alpha_1 = \alpha$, $\alpha_2 = \beta$, $\rho_1 = \rho_2 = \frac{1}{2^{k-1}}$, $h(n) = 1$ and $\delta = 0$, we get 
$$
\Pr\left(X_a \leq \frac{\log n}{\log \log n}\right) = \exp\left(-(1+o_n(1))I^* \cdot \log n\right),
$$
where 
\begin{eqnarray*}
I^* &=& \max_{t \geq 0} \sum_{i=1}^2 \alpha_i \rho_i (1-e^{-c_i t}) \\
&=& \max_{t \geq 0} \frac{1}{2^{k-1}} \left(\alpha (1-e^{-t}) + \beta (1-e^t)\right).
\end{eqnarray*}
The maximum is attained at $t^* = \frac{1}{2} \log \left(\frac{\alpha}{\beta}\right) > 0$ and
$$
I^* = \frac{1}{2^{k-1}} \left(\sqrt{\alpha}-\sqrt{\beta}\right)^2 = 1-\epsilon.
$$
Hence, 
$$
\Pr\left(X_a \leq \frac{\log n}{\log \log n}\right)\geq n^{-1+\epsilon-o_n(1)}
$$
as desired.

\subsection{Proof of the tail bound in Theorem \ref{thm:tailbound2}}
\label{sec:tailbound2}

We recall that $X_a$ is defined as
$$
X_a = \sum_{e: e \cap U = \{a\}} (A_\cH)_a \left(\sum_{ij \subseteq e} \sigma_i \sigma_j - \sigma^{(a)}_i \sigma^{(a)}_j\right).
$$
By definition of $\sigma^{(a)}$, we have
\begin{eqnarray*}
\left(\sum_{ij \subseteq e} \sigma_i \sigma_j - \sigma^{(a)}_i \sigma^{(a)}_j\right) &=& 2\sigma_a \sum_{i \in e \setminus \{a\}} \sigma_i \\
&=& 2\left(k-1-2|e \cap B|\right).
\end{eqnarray*}
Hence, $X_a = \sum_{r=0}^{k-1} c_r Y_r$ where $c_r = 2(k-1-2r)$ and $Y_r \sim \mathrm{Bin}(N_r, p_r)$ with
$$
N_r = \binom{n/2-|U_A|}{k-1-r} \binom{n/2-|U_B|}{r} = (1+o_n(1)) \frac{1}{2^{k-1}} \binom{k-1}{r} \binom{n}{k-1}
$$
and
$$
p_r = \begin{cases} p & \text{if $r=0$} \\ q & \text{otherwise.} \end{cases}
$$

Using Theorem \ref{thm:tail} with 
$$
c_r = 2(k-1-2r), \quad \rho_r = \frac{1}{2^{k-1}}\binom{k-1}{r}, \quad \alpha_r = \begin{cases} \alpha & \text{if $r=0$} \\ \beta & \text{otherwise,} \end{cases}
$$
and $h(n) = 1$ and $\delta = 0$, we have
$$
\Pr\left(X_a \leq -\frac{2\log n}{\log \log n}\right) = \exp\left(-(1+o_n(1))I_2 \cdot \log n\right)
$$
where
\begin{eqnarray*}
I_2 &=& \max_{t \geq 0} \sum_{r=0}^{k-1} \alpha_i \rho_i (1-e^{-c_i t})\\
&=& \max_{t \geq 0} \frac{1}{2^{k-1}} \left(\alpha(1-e^{-(k-1)t}) + \beta\sum_{r=1}^{k-1} \binom{k-1}{r} (1-e^{-(k-1-2r)t})\right),
\end{eqnarray*}
as desired.

\section{Miscellaneous proofs}

\subsection{Proof of Proposition \ref{prop:MLE}}
\label{sec:propMLE}
Recall that the MLE $\widehat{\sigma}_{MLE}(H)$ is defined as
$$
\widehat{\sigma}_{MLE}(H) = \argmax_{x \in \{\pm 1\}^V: \one^T x = 0} f_H(x),
$$
where $f_H(x) = \log \Pr_{(\sigma,\cH)}(\cH=H|\sigma=x)$.
Note that
\begin{eqnarray*}
f_H(x) &=& \log \Pr(\cH = H|\sigma = x) \\
&=& \log \left( \prod_{e \in \binom{V}{k}} \Pr(e \in E(\cH)|\sigma=x)^{(A_H)_e} \Pr(e \not\in E(\cH)|\sigma=x)^{1-(A_H)_e}\right) \\
&=& \sum_{\substack{e \in \binom{V}{k} \\ e\text{: in-cl. w.r.t. $x$}}} (A_H)_e \log p + (1-(A_H)_e) \log (1-p) \\
& & + \sum_{\substack{e \in \binom{V}{k} \\ e\text{: cr.-cl. w.r.t. $x$}}} (A_H)_e \log q + (1-(A_H)_e) \log (1-q) \\
&=& C + \log \left(\frac{p}{1-p}\right) \inprod{A_H, x^{\oeq k}} + \log \left(\frac{q}{1-q}\right) \inprod{A_H, 1-x^{\oeq k}},
\end{eqnarray*}
with
\begin{eqnarray*}
C &=& \log (1-p) \cdot \#(e:\text{ in-cl. w.r.t. $x$}) + \log (1-q) \cdot \#(e:\text{ cr.-cl. w.r.t. $x$}) \\
&=& 2\binom{n/2}{k} \log (1-p) + \left(\binom{n}{k}- 2\binom{n/2}{k}\right) \log (1-q).
\end{eqnarray*}
We note that $C$ is a constant not depending on $x$. Also, $\inprod{A_H, 1}$ is independent of $x$. We get
$$
\widehat{\sigma}_{MLE}(H) = \argmax_{x \in \{\pm 1\}^V: \one^T x = 0} \log \left(\frac{p(1-q)}{q(1-p)}\right) \inprod{A_H, x^{\oeq k}}.
$$
It implies that 
$$
\widehat{\sigma}_{MLE}(H) = \begin{cases}
\displaystyle \argmax_{x \in \{\pm 1\}^V: \one^T x = 0} \inprod{A_H, x^{\oeq k}} & \text{if $p > q$} \\
\displaystyle \argmin_{x \in \{\pm 1\}^V: \one^T x = 0} \inprod{A_H, x^{\oeq k}} & \text{if $p < q$} 
\end{cases}
$$
since $\log \left(\frac{p(1-q)}{q(1-p)}\right)$ is positive if $p > q$ and it is negative if $p < q$. 

\subsection{Proof of Lemma \ref{lemma:variance}}
\label{sec:variance}
We recall that 
$$
Y_e = \left((\one_e^T \sigma_e)\diag(\sigma_e) - \sigma_e\sigma_e^T\right)\Pi\left((\one_e^T \sigma_e)\diag(\sigma_e) - \sigma_e\sigma_e^T\right),
$$
where $\Pi = I - \frac{1}{n}\sigma\sigma^T - \frac{1}{n}{\one\one^T}$. We defined $\Sigma$ as
$$
\Sigma = \sum_{e} \E (A_\cH)_e \cdot Y_e = p\sum_{e:e \text{ is in-cl.}} Y_e + q\sum_{e:e \text{ is cross-cl.}} Y_e.
$$
\begin{lemma}[Lemma \ref{lemma:variance}]
$\Sigma = c_1 \cdot \frac{1}{n} \sigma\sigma^T + c_2 \cdot \Pi$ where
$$
c_1 = (k-1)\beta \log n + O\left(\frac{\log n}{n}\right)
$$
and
$$
c_2 =  2(k-1)\left(\frac{k}{2^{k}} \alpha+\left(1-\frac{k}{2^k}\right)\beta\right)\log n + O\left(\frac{\log n}{n}\right).
$$
\end{lemma}

We note that $\sum_{e: e\text{ is in-cl.}} Y_e$ and $\sum_{e: e\text{ is cross-cl.}}Y_e$ are invariant under any permutation on $V$ preserving $\sigma$. It implies that the both matrices and $\Sigma$ are in the span of $\Pi$, $\sigma\sigma^T$ and $\one\one^T$. Moreover,
$$
\left((\one_e^T \sigma_e)\diag(\sigma_e) - \sigma_e\sigma_e^T\right)\one = (\one_e^T \sigma_e) \sigma_e - (\one^T \sigma_e) \sigma_e = 0
$$
so $\inprod{Y_e, \one\one^T} = 0$. Hence, 
$$
\Sigma = \frac{\inprod{\Sigma, \Pi}}{\inprod{\Pi, \Pi}} \Pi + \inprod{\Sigma, \frac{1}{n}\sigma\sigma^T} \cdot \frac{1}{n} \sigma\sigma^T.
$$
It implies that 
\begin{eqnarray*}
c_1 &=& \inprod{\Sigma, \frac{1}{n}\sigma\sigma^T} =
\frac{1}{n} \sigma^T \Sigma \sigma \\
c_2 &=& \frac{\inprod{\Sigma, \Pi}}{\inprod{\Pi, \Pi}} = \frac{1}{n-2}\left(tr(\Sigma) - \frac{1}{n} \sigma^T \Sigma \sigma\right) 
\end{eqnarray*}

Now, let us first compute $\sigma^T \Sigma \sigma$. For simplicity, let $r = \frac{1}{2}(k-\one_e^T \sigma_e)$. Then,
\begin{eqnarray*}
\sigma^T Y_e \sigma &=& \sigma^T\left((k-2r)\diag(\sigma_e) - \sigma_e\sigma_e^T\right)\Pi\left((k-2r)\diag(\sigma_e) - \sigma_e\sigma_e^T\right)\sigma \\
&=&
\left((k-2r) \one_e - k\sigma_e\right)^T \Pi \left((k-2r) \one_e - k\sigma_e\right) \\
&=& \|(k-2r) \one_e - k\sigma_e\|_2^2 - \frac{1}{n} ((k-2r) \one_e - k\sigma_e)^T\sigma)^2 \\
&=& (k^3-k(k-2r)^2) - \frac{1}{n} \left((k-2r)^2 - k^2\right)^2\\
&=& 4kr(k-r) - \frac{16}{n} r^2(k-r)^2.
\end{eqnarray*}
In particular, $\sigma^T Y_e\sigma = 0$ if $e$ is in-cluster with respect to $\sigma$. Hence,
\begin{eqnarray*}
\sigma^T \Sigma \sigma &=& q\sum_{e:e \text{ is cross-cl.}} \sigma^T Y_e\sigma \\
&=& q \sum_{r=1}^{k-1} \left(4kr(k-r) - \frac{16}{n} r^2(k-r)^2\right)\binom{n/2}{r}\binom{n/2}{k-r}.
\end{eqnarray*}
We note that for any $s, t \in \{1, 2\}$, we have
\begin{eqnarray*}
\sum_{r=1}^{k-1} \binom{r}{s}\binom{k-r}{t} \binom{n/2}{r}\binom{n/2}{k-r} &=& 
\binom{n/2}{s}\binom{n/2}{t} \binom{n-s-t}{k-s-t} \\
&=& \binom{n/2}{s}\binom{n/2}{t} \binom{n}{k}\binom{k}{s+t}\binom{n}{s+t}^{-1}.
\end{eqnarray*}
Hence, 
\begin{eqnarray*}
\sum_{r=1}^{k-1} r(k-r) \binom{n/2}{r}\binom{n/2}{k-r}
&=& \binom{n}{k} \frac{\binom{k}{2}\binom{n/2}{1}^2}{\binom{n}{2}} \\
&=& \binom{n}{k} \frac{k(k-1)n^2}{4n(n-1)} = \frac{k(k-1)}{4} \binom{n}{k} + O(n^{k-1}).
\end{eqnarray*}
and
$$
\frac{1}{n} \sum_{r=1}^{k-1} r^2(k-r)^2 \binom{n/2}{r}\binom{n/2}{k-r} = O(n^{k-1}). 
$$
Hence, 
$$
c_1 = \frac{1}{n} \sigma^T \Sigma\sigma = (k-1)\beta \log n + O\left(\frac{\log n}{n}\right).
$$
On the other hand, 
\begin{eqnarray*}
tr(Y_e) &=& tr\left(\left((k-2r)\diag(\sigma_e) - \sigma_e\sigma_e^T\right)^2\Pi\right) \\
&=& tr\left(\left((k-2r)\diag(\sigma_e) - \sigma_e\sigma_e^T\right)^2\right)
-\frac{1}{n}\left\|\left((k-2r)\diag(\sigma_e) - \sigma_e\sigma_e^T\right)\sigma\right\|_2^2 \\
&=& \left((k-2)(k-2r)^2+k^2\right)-\frac{1}{n}\|(k-2r)\one_e-k\sigma_e\|_2^2 \\
&=& \left((k-2)(k-2r)^2+k^2\right)-\frac{1}{n}(k^3-k(k-2r)^2) \\
&=& (k^3-k^2) - 4(k-2+\frac{k}{n})r(k-r),
\end{eqnarray*}
so we have
\begin{eqnarray*}
tr(\Sigma) &=& (k^3-k^2)\left(q\binom{n}{k}+2(p-q)\binom{n/2}{k}\right) - 4q\left(k-2+\frac{k}{n}\right) \cdot \frac{n^2}{4} \binom{n-2}{k-2} \\
&=& \left[(k^2-k)\left(\beta+\frac{\alpha-\beta}{2^{k-1}}\right)-(k-1)(k-2)\beta\right]n\log n + O(\log n).
\end{eqnarray*}
and hence
\begin{eqnarray*}
c_2 &=& \frac{1}{n-2}\left(tr(\Sigma)-\frac{1}{n} \sigma^T \Sigma \sigma\right) \\
&=& 2(k-1)\left(\frac{k}{2^{k}} \alpha+\left(1-\frac{k}{2^k}\right)\beta\right)\log n + O\left(\frac{\log n}{n}\right).
\end{eqnarray*}

\end{document}